\documentclass[preprint,11pt]{article}

\usepackage{amssymb,amsfonts,amsmath,amsthm,amscd,dsfont,mathrsfs}
\usepackage{eqsection,graphicx,float,psfrag,epsfig,color,url}

\newtheorem{theorem}{Theorem}[section]

\newtheorem{lemma}[theorem]{Lemma}

\newtheorem{proposition}[theorem]{Proposition}
\newtheorem{corollary}[theorem]{Corollary}
\newtheorem{definition}{Definition}

\definecolor{Red}{rgb}{1,0,0}
\definecolor{Blue}{rgb}{0,0,1}
\definecolor{Olive}{rgb}{0.41,0.55,0.13}
\definecolor{Green}{rgb}{0,1,0}
\definecolor{MGreen}{rgb}{0,0.8,0}
\definecolor{DGreen}{rgb}{0,0.55,0}
\definecolor{Yellow}{rgb}{1,1,0}
\definecolor{Cyan}{rgb}{0,1,1}
\definecolor{Magenta}{rgb}{1,0,1}
\definecolor{Orange}{rgb}{1,.5,0}
\definecolor{Violet}{rgb}{.5,0,.5}
\definecolor{Purple}{rgb}{.75,0,.25}
\definecolor{Brown}{rgb}{.75,.5,.25}
\definecolor{Grey}{rgb}{.5,.5,.5}
\definecolor{Pink}{rgb}{1,0,1}
\definecolor{DBrown}{rgb}{.5,.34,.16}

\definecolor{Black}{rgb}{0,0,0}

\def\projpar{{\rm P}_{\parallel}}
\def\projperp{{\rm P}_{\perp}}
\def\Apar{A_{\parallel}}
\def\Aperp{A_{\perp}}
\def\im{{\rm im}}
\def\cP{{\cal P}}
\def\tA{\tilde{A}}
\def\normal{{\sf N}}

\def\ve{\varepsilon}
\def\prob{{\mathbb P}}

\def\Sco{\overline{S}}
\def\ind{{\mathbb I}}
\def\hx{\widehat{x}}
\def\argmin{{\rm argmin}}
\def\seF{{\sf F}}

\def\cost{{\cal C}}

\def\sign{\mbox{\rm sign}}

\def\subg{\mbox{\rm sg}}

\def\smaxA{\sigma_{\max}(A)}

\def\la{\lambda}
\def\supp{{\rm supp}}
\def\eps{\varepsilon}
\def\dxparal{r^{\parallel}}

\def\MSE{{\rm MSE}}

\def\reals{{\mathbb R}}
\def\naturals{{\mathbb N}}
\def\<{\langle}
\def\>{\rangle}
\def\E{{\mathbb E}}
\def\identity{{\mathbf I}}

\def\finite{{\sf B}}

\def\AMP{{\textrm{\rm AMP}}}

\def\onsager{\omega}

\def\tauinf{{\tau_*}}

\def\de{{\rm d}}
\def\LASSO{{\rm LASSO}}
\def\dx{r}
\def\covz{{\sf R}}
\def\cE{{\cal E}}
\def\ra{{\sf r}}
\def\G{{\sf G}}
\def\cA{{\cal A}}
\def\ed{\stackrel{{\rm d}}{=}}
\def\tZ{\tilde{Z}}
\def\tcovz{\tilde{\covz}}
\def\va{\vec{a}}
\def\sX{{\sf X}}
\newcommand{\empr}{\hat{p}}

\newcommand{\tQ}{\tilde{Q}}

\newcommand{\f}{\frac}
\newcommand{\pl}{\parallel}
\newcommand{\deq}{\stackrel{\text{\rm d}}{=}}

\newcommand{\sigal}{\mathfrak}
\newcommand{\almostsurely}{{\rm a.s.}}
\newcommand{\asequal}{\stackrel{\almostsurely}{=}}

\newcommand{\pW}{p_{W}}
\newcommand{\eprooft}{\hfill$\Box$}
\newcommand{\order}{\vec{o}}
\newcommand{\tM}{\tilde{M}}
\def\lbq{\rho}
\def\lbm{\varsigma}
\def\hsigmamin{\hat{\sigma}_{\rm min}}

\begin{document}

\title{The LASSO risk for gaussian matrices}

\author{Mohsen Bayati\thanks{Department of Electrical Engineering,
Stanford University} \;\;\; and\;\;\;
Andrea Montanari${}^{*,}$\thanks{Department of Statistics,
Stanford University}}

\date{}

\maketitle

\begin{abstract}
We consider the problem of learning a coefficient vector $x_0\in\reals^N$
from noisy linear observation $y=Ax_0+w\in\reals^n$. In many contexts
(ranging from model selection to image processing)  it is desirable
to construct a sparse estimator $\hx$. In this case, a popular
approach consists in solving an $\ell_1$-penalized least squares problem
known as the LASSO or Basis Pursuit DeNoising (BPDN).

For sequences of matrices $A$ of increasing dimensions, with
independent gaussian entries, we prove that the normalized
risk of the LASSO converges to a limit, and we obtain
an explicit expression for this limit.
Our result is the first rigorous
derivation of an explicit formula for the asymptotic
mean squared error of the LASSO for random instances.
The proof technique is based on the analysis of $\AMP$, a recently
developed efficient algorithm, that is inspired from graphical models ideas.

Simulations on real data matrices suggest that our results can be
relevant in a broad array of practical applications.
\end{abstract}
%
%
\section{Introduction}\label{sec:intro}

Let $x_0\in\reals^N$ be an unknown vector, and assume that
a vector $y\in\reals^n$ of noisy linear measurements of $x_0$ is available.
The problem of reconstructing $x_0$ from such measurements
arises in a number of disciplines, ranging from statistical learning
to signal processing.
In many contexts the measurements are modeled by
\begin{eqnarray}
y = Ax_0+w\, ,
\end{eqnarray}
where $A\in\reals^{n\times N}$ is a known measurement matrix,
and $w$ is a noise vector.

The LASSO or Basis Pursuit Denoising (BPDN) is a method
for reconstructing the unknown vector $x_0$ given $y$, $A$,
and is particularly useful when one seeks sparse solutions.
For given $A$, $y$, one  considers the cost functions
$\cost_{A,y}:\reals^{N}\to\reals$ defined by
\begin{eqnarray}
\cost_{A,y}(x) = \frac{1}{2}\,\|y-Ax\|^2 + \,\lambda \|x\|_1\, ,\label{eq:LASSO-Problem}
\end{eqnarray}
with $\lambda>0$.
The original signal is estimated by
\begin{eqnarray}
\hx(\lambda;A,y) = \argmin_x\, \cost_{A,y}(x)\, .\label{eq:LassoOPT}
\end{eqnarray}
In what follows we shall often omit the arguments $A,y$
(and occasionally $\lambda$) from the above
notations. We will also use $\hx(\lambda;N)$ to emphasize the $N$-dependence.
Further $\|v\|_p \equiv (\sum_{i=1}^mv_i^p)^{1/p}$
denotes the $\ell_p$-norm of a vector $v\in \reals^m$
(the subscript $p$ will often be omitted if $p=2$).

A large and rapidly growing literature is devoted to
developing fast algorithms for solving the optimization
problem (\ref{eq:LassoOPT}) and characterizing the performances
and optimality of the estimator $\hx$. We refer to Section \ref{sec:Related} for an unavoidably incomplete overview.

Despite such substantial effort, and many remarkable achievements,
our understanding of (\ref{eq:LassoOPT}) is not even comparable to
the one we have of more classical topics in statistics and estimation
theory. For instance, the best bound on the mean squared error ($\MSE$)
of the estimator (\ref{eq:LassoOPT}), i.e. on
the quantity $N^{-1}\|\hx-x_0\|^2$, was proved by  Candes, Romberg and Tao
\cite{CandesStable}
(who in fact did not consider the LASSO but a related optimization problem).
Their result estimates the
mean squared error only up to an unknown numerical multiplicative factor.
Work by Candes and Tao \cite{Dantzig} on the analogous
\emph{Dantzig selector}, upper bounds
the mean squared error up to a factor $C\log N$,
under somewhat different assumptions.

The objective of this paper is to complement this type of
`rough but robust' bounds by proving
\emph{asymptotically exact} expressions for the mean square
error. Our asymptotic result holds almost
surely for sequences of random matrices $A$ with fixed aspect
ratio and independent gaussian entries. While this setting
is admittedly specific, the careful study of such matrix ensembles
has a long tradition both in statistics and communications theory and
has spurred many insights \cite{JohnstoneICM,Telatar}.
Further, we carried out simulations on real data matrices
with continuous entries (gene expression data) and binary feature matrices
(hospital medical records). The results appear to be quite encouraging.

Although our rigorous results
are asymptotic in the problem dimensions,
numerical simulations have shown that they are accurate
already on problems with a few hundreds of variables.
Further, they seem to enjoy a remarkable \emph{universality}
property and to hold for a fairly broad family of matrices
\cite{NSPT}.
Both these phenomena are analogous to ones in random matrix theory,
where delicate asymptotic properties of gaussian ensembles
were subsequently proved to hold for much broader classes of random
matrices. Also, asymptotic statements
in random matrix theory have been replaced over time
by concrete probability bounds in finite dimensions.
Of course the optimization problem (\ref{eq:LASSO-Problem})
is not immediately related to spectral properties of the random matrix $A$.
As a consequence, universality and non-asymptotic results in random matrix
theory cannot be directly exported to the present problem. Nevertheless,
we expect such developments to be foreseeable.

Our proofs are based on the analysis of an efficient iterative
algorithm first proposed by \cite{DMM09}, and called AMP,
for approximate message passing.
The algorithm is inspired by  belief-propagation on graphical models;
although the resulting iteration is significantly simpler
(and scales linearly in the number of nodes).
Extensive simulations \cite{NSPT} showed
that, in a number of settings, $\AMP$ performances are statistically
indistinguishable to the ones of LASSO,
while its complexity is essentially as low as the one of the simplest
greedy algorithms.

The proof technique just described is new. Earlier literature analyzes
the convex optimization problem (\ref{eq:LassoOPT}) --or similar problems--
by a clever construction of an approximate optimum, or of a dual witness.
Such constructions are largely explicit. Here instead we
prove an asymptotically exact characterization of a rather non-trivial
iterative algorithm. The algorithm is then proved to converge to the
exact optimum.
%
%
\subsection{Definitions}

In order to define the $\AMP$ algorithm,
we denote by $\eta:\reals\times\reals_+\to\reals$
the soft thresholding function
\begin{eqnarray}
\label{eq:eta-def}
\eta(x;\theta) = \left\{
\begin{array}{ll}
x-\theta & \mbox{if $x>\theta$,}\\
0 &  \mbox{if $-\theta\le x\le\theta$,}\\
x+\theta & \mbox{otherwise.}
\end{array}\right.
\end{eqnarray}
The  algorithm constructs a sequence of estimates
$x^t\in\reals^N$, and residuals $z^t\in\reals^n$,
according to the iteration
\begin{align}
x^{t+1}&=\eta(A^*z^t+x^t;\theta_t) ,\label{eq:dmm}\\
z^t &= y - Ax^t+\f{1}{\delta} z^{t-1}
\left\<\eta'(A^*z^{t-1}+x^{t-1};\theta_{t-1})\right\>\, ,\nonumber
\end{align}
initialized with $x^0=0\in \reals^N$. Here $A^*$ denotes the transpose of matrix $A$, $\delta\equiv n/N$,
and $\eta'(\,\cdot\,;\,\cdot\,)$ is the derivative of
the soft thresholding function with respect to its first argument.
Given a scalar function $f$ and a vector $u\in\reals^m$, we
let $f(u)$ denote the vector $(f(u_1),\dots,f(u_m))\in\reals^m$
obtained by applying $f$ componentwise. Finally
$\<u\> \equiv m^{-1}\sum_{i=1}^m u_i$ is the average of the vector
$u\in\reals^m$.

As already mentioned, we will consider sequences of instances
of increasing sizes, along which the LASSO behavior has a non-trivial limit.
\begin{definition}
The sequence of instances $\{x_0(N), w(N), A(N)\}_{N\in\naturals}$
indexed by $N$ is said to be a \emph{converging sequence}
if
$x_0(N)\in\reals^{N}$, $w(N)\in\reals^n$, $A(N)\in\reals^{n\times N}$
with $n=n(N)$ is such that $n/N\to\delta\in(0,\infty)$,
and in addition the following conditions hold:
\begin{itemize}
\item[$(a)$] The empirical distribution of the entries of $x_0(N)$
converges weakly to a probability measure $p_{X_0}$ on $\reals$
with bounded second moment. Further
$N^{-1}\sum_{i=1}^Nx_{0,i}(N)^2\to \E_{p_{X_0}}\{X_0^{2}\}$.
\item[$(b)$] The empirical distribution of the entries of $w(N)$
converges weakly to a probability measure $p_{W}$ on $\reals$
with bounded second moment. Further
$n^{-1}\sum_{i=1}^nw_{i}(N)^2\to \E_{p_{W}}\{W^{2}\}$.
\item[$(c)$]If $\{e_i\}_{1\le i\le N}$, $e_i\in\reals^N$ denotes the standard
basis, then $\max_{i\in [N]}\|A(N)e_i\|_2$, $\min_{i\in [N]}\|A(N)e_i\|_2\to 1$,
as $N\to\infty$ where $[N]\equiv\{1,2,\ldots,N\}$.
\end{itemize}
\end{definition}
Let us stress that our proof only applies to a subclass of
converging sequences, namely for gaussian measurement matrices
$A(N)$. The notion of converging sequences is however important
since it defines a class of problem instances to which the
ideas developed below might be generalizable.
Also, while the  measurement matrices $A(N)$ will be random, the
signal $x_0(N)$, and noise vectors $w(N)$ will be deterministic.

For a converging sequence of instances,
and an arbitrary sequence of thresholds $\{\theta_t\}_{t\ge 0}$
(independent of $N$), the asymptotic behavior of the recursion
(\ref{eq:dmm}) can be characterized as follows.

Define the sequence $\{\tau_t^2\}_{t\ge 0}$
by setting  $\tau_{0}^2 =\sigma^2+\E\{X_0^2\}/\delta$
(for $X_0\sim p_{X_0}$ and $\sigma^2\equiv \E\{W^2\}$, $W\sim p_W$)
and letting, for all $t\ge 0$:
\begin{eqnarray}
\tau_{t+1}^2 & = & \seF(\tau_t^2,\theta_t)\, ,\label{eq:1-dim-SE}\\
\seF(\tau^2,\theta) &\equiv &\sigma^2+\frac{1}{\delta}\,
\E\{\,[\eta(X_0+\tau Z;\theta)-X_0]^2\}\,,
\end{eqnarray}
where $Z\sim\normal(0,1)$ is independent of $X_0$. Notice that the
function $\seF$ depends implicitly on the law $p_{X_0}$. We will see
later that the quantity $A^*z^t+x^t$ has the same distribution as
$X_0+\tau_tZ$. In other words, $\tau_t^2$ is the MSE
of the estimator $A^*z^t+x^t$ for $x_0$.

We say a function $\psi:\reals^2\to\reals$ is \emph{pseudo-Lipschitz} if there exist a constant $L>0$ such that for all $x,y\in\reals^2$: $|\psi(x)-\psi(y)|\le L(1+\|x\|_2+\|y\|_2)\|x-y\|_2$.
(This is a special case of the definition used
in \cite{BM-MPCS-2010} where such a function is called pseudo-Lipschitz
\emph{of order 2}.)

The next proposition that was conjectured in \cite{DMM09} and proved in
\cite{BM-MPCS-2010} shows that the behavior of $\AMP$ can be tracked by the
above one dimensional recursion. We
often refer to this prediction by \emph{state evolution}.

\begin{theorem}[\cite{BM-MPCS-2010}]\label{prop:state-evolution} Let $\{x_0(N), w(N), A(N)\}_{N\in\naturals}$ be a converging sequence of
instances with the entries of $A(N)$ iid normal with mean $0$ and variance
$1/n$ and let $\psi:\reals\times\reals\to \reals$ be a pseudo-Lipschitz function. Then, almost surely
\begin{eqnarray}\label{eq:state-evolution}
\lim_{N\to\infty}\frac{1}{N}\sum_{i=1}^N\psi
\big(x_{i}^{t+1},x_{0,i}\big) = \E\Big\{\psi\big(\eta(X_0+\tau_t Z;\theta_t),X_0\big)\Big\}\, ,
\end{eqnarray}
where $Z\sim\normal(0,1)$ is independent of $X_0\sim p_{X_0}$.
\end{theorem}
In order to establish the connection with the LASSO,
a specific policy has to be chosen for the thresholds
$\{\theta_t\}_{t\ge 0}$. Throughout this paper we
will take $\theta_t = \alpha\tau_t$ with $\alpha$ is fixed.
In other words, the sequence $\{\tau_t\}_{t\ge 0}$ is
given by the recursion
\begin{eqnarray}
\tau_{t+1}^2 = \seF(\tau_t^2,\alpha\tau_t)\, .
\end{eqnarray}
This choice enjoys several convenient properties \cite{DMM09}. In
particular the sequence $\{\tau_t\}$ always converges to the largest
solution of the fixed point equation $\tau^2=\seF(\tau^2,\alpha\tau)$.
Further, it is a very natural choice from an intuitive point of view.
Consider indeed the AMP recursion (\ref{eq:dmm}). At each step we
construct a vector of `effective observations' $y^t=x^t+A^*z^t\in\reals^N$. This
can be regarded as a noisy version of the signal $x_0$, whereby each
entry of $x_0$ has been corrupted by Gaussian noise with mean $0$ and
variance $\tau^2_t$. Indeed, as
witnessed by Theorem \ref{prop:state-evolution}, $y^t$ is
asymptotically distributed as $x_0+w_t$ with $w^t\sim
\normal(0,\tau_t\identity_{N\times N})$ (this statement holds in the sense
of finite-dimensional marginals).
Hence, it is very natural to obtain a refined estimate by applying
the soft thresholding denoiser $\eta(\,\cdot\,;\theta_t)$
componentwise to $y_t$, which is exactly what happens in the first
equation in (\ref{eq:dmm}).  This denoiser shrinks component $y^t_{i}$
to $0$ if $|y^t_{i}|\le \theta_t$. The interpretation is that any
entry above $\theta_t$ is regarded as pure noise. Obviously this
suggests to choose $\theta_t$ proportional to the standard deviation
of the effective noise, $\tau_t$. This is indeed confirmed by a
careful mathematical analysis: choosing $\theta_t = \alpha\tau_t$ is
minimax optimal, for a suitable choice of the proportionality constant
$\alpha$ \cite{DJ94a,DJ98,DMM09}.

Let us finally discuss why there should be any relation at all between
the AMP algorithm (\ref{eq:dmm}) and the solution of the LASSO.
Assume that $\theta_t\to\theta$, and that  $(x,z)$ is a fixed point of
the corresponding AMP iteration. Let $\onsager =
\delta^{-1}\<\eta'(x+A^*z;\theta)\>$. Then the fixed point condition
reads
\begin{eqnarray}
x & = & \eta(x+A^*z;\theta)\, ,\\
z &=& y-Ax+\onsager\, z\, .
\end{eqnarray}
Notice that $x = \eta(r;\theta)$ if and only if there exists
$v(x)\in\partial\|x\|_1$  such that $x+\theta v(x) = r$ (here
$\partial f$ denotes the subgradient of the function $f$). It follows
that the fixed point condition can be rewritten as
\begin{eqnarray}
A^*(y-Ax) =\theta (1-\omega)\, v\, ,\;\;\;\;\;\;\;
v\in\partial\|x\|_1\, .
\end{eqnarray}
Comparing with the stationarity condition for the LASSO cost function
(\ref{eq:LASSO-Problem}) we obtain the following.
\begin{lemma}\label{lemma:Correspondence}
Any fixed point $x^t=x$ of the AMP iteration with
$\theta_t=\theta$ is a minimizer of
the LASSO cost function with
\begin{eqnarray}
\lambda = \theta\Big\{1-
\frac{1}{\delta}\<\eta'(x+A^*z;\theta)\>\Big\}\, .
\end{eqnarray}
\end{lemma}
%
%
\subsection{Main result}\label{sec:Results}

Before stating our results,
we have to describe a \emph{calibration} mapping between $\alpha$
and $\lambda$ that was introduced in \cite{NSPT}. This mapping is necessary since in the analysis of AMP $\alpha$ plays the role of $\lambda$. In other words, it can be viewed as regularization parameter and controls sparsity of AMP estimates. In particular, we will show that there exist a one-to-one (monotone) function between values of $\alpha$ and $\lambda$.

\subsubsection{Calibration between $\alpha$ and $\lambda$}
Let us start by stating some convenient properties of
the state evolution recursion.
\begin{proposition}[\cite{DMM09}]\label{propo:UniqFP}
Let $\alpha_{\rm min}= \alpha_{\rm min}(\delta)$ be the unique
non-negative solution of the equation
\begin{eqnarray}
(1+\alpha^2)\Phi(-\alpha)-\alpha\phi(\alpha) = \frac{\delta}{2}\, ,
\label{eq:AlphaMin}
\end{eqnarray}
with $\phi(z) \equiv e^{-z^2/2}/\sqrt{2\pi}$ the standard gaussian density
and $\Phi(z) \equiv\int_{-\infty}^{z}\phi(x)\,\de x$.

For any $\sigma^2>0$, $\alpha>\alpha_{\rm min}(\delta)$,
the  fixed point equation
$\tau^2 = \seF(\tau^2,\alpha\tau)$ admits a unique solution.
Denoting by $\tau_*=\tau_*(\alpha)$ this solution, we
have $\lim_{t\to\infty}\tau_t=\tau_*(\alpha)$. Further the convergence takes
place for any initial condition and is monotone. Finally
$\left|\frac{\de \seF}{\de\tau^2}(\tau^2,\alpha\tau)\right|<1$
at $\tau=\tau_*$.
\end{proposition}
For greater convenience of the reader, a proof of this statement is
provided in Appendix \ref{app:UniqFP}.

We then define the function $\alpha\mapsto \lambda(\alpha)$
on $(\alpha_{\rm min}(\delta),\infty)$, by
\begin{eqnarray}
\lambda(\alpha) \equiv \alpha\tauinf\left[1 - \frac{1}{\delta}
\E\big\{\eta'(X_0+\tauinf Z;\alpha\tauinf)\big\}\right]\, .\label{eq:calibration}
\end{eqnarray}
This function defines a correspondence (calibration) between the threshold $\alpha\tau_*$ and the regularization parameter
$\lambda$. It should be intuitively clear that larger $\lambda$
corresponds to larger thresholds and hence larger $\alpha$
since both cases yield smaller estimates of $x_0$.
The specific choice in Eq.~(\ref{eq:calibration}) is motivated by
Lemma \ref{lemma:Correspondence}.

In the following we will need to invert this function.
We thus define $\alpha:(0,\infty)\to(\alpha_{\rm min},\infty)$
in such a way that
\begin{eqnarray}
\alpha(\lambda) \in \big\{\, a\in (\alpha_{\rm min},\infty)\, :\,
\lambda(a) =\lambda\big\}\, .\label{eq:AlphaOfLambda}
\end{eqnarray}
The next result implies that the set on the right-hand side
is non-empty and therefore the function $\lambda\mapsto\alpha(\lambda)$
is well defined.
\begin{proposition}[\cite{NSPT}]\label{propo:Lambda}
The function  $\alpha\mapsto\lambda(\alpha)$ is
continuous on the interval
$(\alpha_{\rm min},\infty)$ with $\lambda(\alpha_{\rm min}+)= -\infty$
and $\lim_{\alpha\to\infty}\lambda(\alpha) = \infty$.

Therefore the function $\lambda\mapsto\alpha(\lambda)$
satisfying Eq.~(\ref{eq:AlphaOfLambda}) exists.
\end{proposition}
A proof of this statement is provided in Section \ref{app:Lambda}.
We will denote by $\cA = \alpha((0,\infty))$ the image of the function
$\alpha$. Notice that the definition of $\alpha$ is \emph{a priori}
not unique. We will see that uniqueness follows from our main theorem.

Examples of the mappings $\tau^2\mapsto\seF(\tau^2,\alpha\tau)$,
$\alpha\mapsto\tau_*(\alpha)$ and $\alpha\mapsto\lambda(\alpha)$
are presented in Figures \ref{fig:tau2->F(tau2)}, \ref{fig:alpha2taustar}, and \ref{fig:alpha2lam} respectively.

\subsubsection{Main results}

We can now state our main result.
\begin{theorem}\label{thm:Risk}
Let $\{x_0(N), w(N), A(N)\}_{N\in\naturals}$ be a converging sequence of
instances with the entries of $A(N)$ iid normal with mean $0$ and variance
$1/n$. Denote by $\hx(\lambda;N)$ the \LASSO\, estimator for
instance $(x_0(N), w(N), A(N))$, with $\sigma^2,\lambda>0$,
$\prob\{X_0\neq 0\}>0$ and
let $\psi:\reals\times\reals\to \reals$ be a pseudo-Lipschitz function.
Then, almost surely
\begin{eqnarray}\label{eq:asymptotic-result}
\lim_{N\to\infty}\frac{1}{N}\sum_{i=1}^N\psi
\big(\hx_{i},x_{0,i}\big) = \E\Big\{\psi\big(\eta(X_0+\tau_* Z;\theta_*),X_0\big)\Big\}\, ,
\end{eqnarray}
where $Z\sim\normal(0,1)$ is independent of $X_0\sim p_{X_0}$,
$\tau_*=\tau_*(\alpha(\lambda))$ and $\theta_*=\alpha(\lambda)
\tau_*(\alpha(\lambda))$.
\end{theorem}
Let us emphasize oonce more that the vectors $x_0(N)$, $w(N)$ are
deterministic in this statement, and `almost surely' is understood
with respect to the choice of $A(N)$.

As a corollary, using function $\psi(a,b)\equiv (a-b)^2$ we obtain:
\begin{corollary}\label{coro:lasso-risk}
Assume the hypothesis of Theorem \ref{thm:Risk}. Let $\hx(\lambda;N)$ be the \LASSO\, estimator for instance $(x_0(N), w(N), A(N))$. Then, almost surely
\begin{eqnarray*}
\lim_{N\to\infty}\frac{1}{N}\|x_0-\hx(\lambda;N)\|^2 = \E\Big\{\left[\eta(X_0+\tau_* Z;\theta_*)-X_0\right]^2\Big\}=\delta(\tau_*^2-\sigma^2)\, ,
\end{eqnarray*}
where $Z\sim\normal(0,1)$ is independent of $X_0\sim p_{X_0}$,
$\tau_*=\tau_*(\alpha(\lambda))$ and $\theta_*=\alpha(\lambda)
\tau_*(\alpha(\lambda))$.
\end{corollary}

As a second corollary of Theorem \ref{thm:Risk}, the function $\lambda\mapsto\alpha(\lambda)$ is
indeed uniquely defined.
\begin{corollary}\label{coro:AlphaUnique}
For any $\lambda,\sigma^2>0$ there exists a unique
$\alpha>\alpha_{\rm min}$ such that
$\lambda(\alpha) = \lambda$ (with the function $\alpha
\to\lambda(\alpha)$ defined as in Eq.~(\ref{eq:calibration}).

Hence the function $\lambda\mapsto\alpha(\lambda)$ is continuous non-decreasing
with $\alpha((0,\infty)) \equiv\cA = (\alpha_0,\infty)$.
\end{corollary}
The proof of this corollary (which uses Theorem \ref{thm:Risk})
is provided in Appendix \ref{app:LambdaBis}.

The assumption of a converging problem-sequence is
important for the result to hold, while the hypothesis of gaussian
measurement matrices $A(N)$ is necessary for the proof technique
to be correct. On the other hand,
the restrictions $\lambda,\sigma^2>0$, and $\prob\{X_0\neq 0\}>0$
(whence $\tau_*\neq0$ using Eq.~\eqref{eq:calibration})
are made in order to avoid technical complications due
to degenerate cases. Such cases can be resolved by continuity arguments.

We prove Theorem \ref{thm:Risk} by proving the following result in Section \ref{sec:MainProof}.
\begin{theorem}\label{thm:FiniteTime}
Assume the hypotheses of Theorem \ref{thm:Risk}.
Let $\hx(\lambda;N)$ be the LASSO estimator for
instance $(x_0(N), w(N), A(N))$, and denote by $\{x^t(N)\}_{t\ge 0}$
the sequence of estimates produced by $\AMP$. Then
\begin{eqnarray}
\lim_{t\to\infty}\lim_{N\to\infty} \frac{1}{N}\|x^t(N)-\hx(\lambda;N)\|_2^2 = 0\, ,
\end{eqnarray}
almost surely.
\end{theorem}
Let us emphasize that the statement of Theorem \ref{thm:FiniteTime}
requires taking the limit of infinite dimensions $N\to\infty$ \emph{before} the
limit of an infinite number of iterations $t\to\infty$. In this sense it is
(informally speaking) a statement about the high-dimensional limit
behavior, for a large-but-finite number of iterations. Although this
is not a common setting within mathematical optimization, we think
that it is particularly compelling from a compressed sensing point of
view. It implies that, for any finite tolerance $\ve>0$, there exists a
finite number of iterations $t_*(\ve)$ such that for any fixed $t\ge
t_*(\ve)$,  AMP has mean squared error at most $\ve$ larger
than the LASSO, \emph{with high probability as $N\to\infty$}.
Further, closer analysis of the state evolution recursion \cite{DMM09,NSPT}
implies that $t_*(\eps) \le C\, \log(1/\ve)$ for some  constant $C$
independent of the dimension, and the signal $x_0$, provided the
under-sampling ratio $\delta$ is larger than a phase transition value
$\delta_c$.
Notice that taking the high dimensional point of view yields us a
considerably faster convergence than the optimum rate at fixed
dimension, namely $t_*(\ve) \le C/\sqrt{\ve}$ \cite{Beck}.

\begin{figure}
\centering
  \includegraphics[width=4.7in]{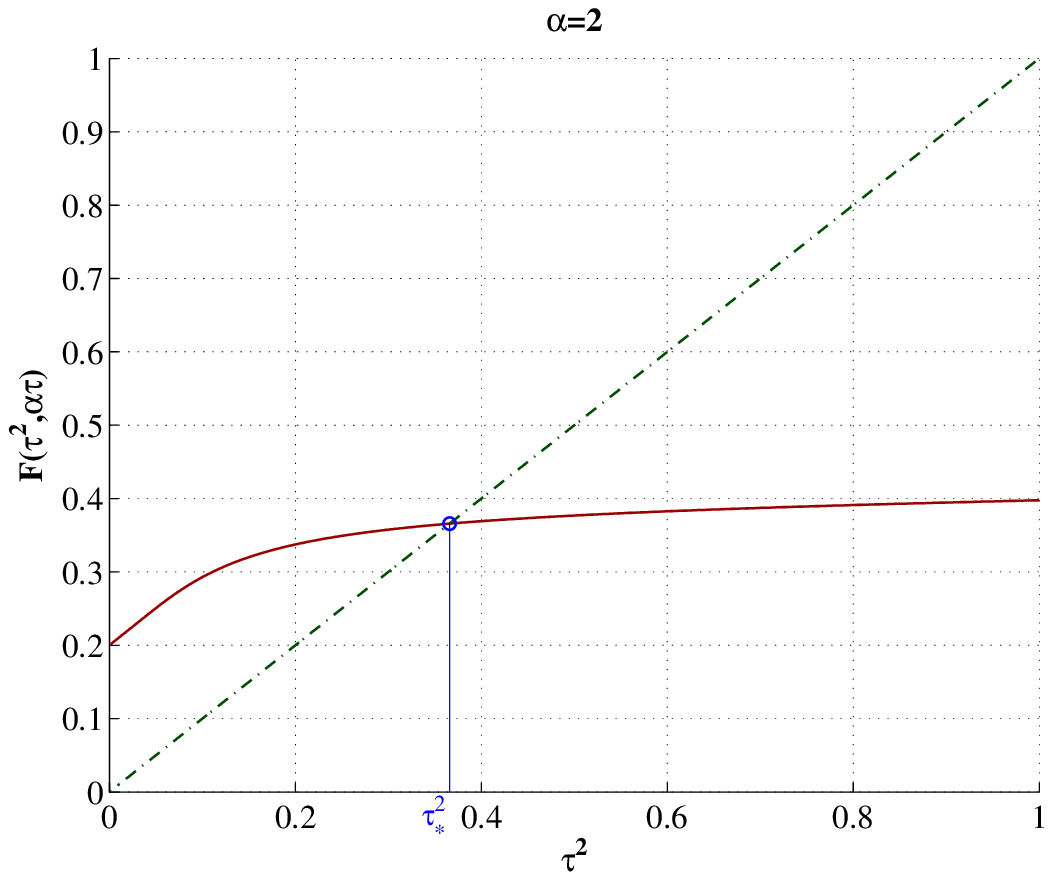}
  \caption{Mapping $\tau^2\mapsto\seF(\tau^2,\alpha\tau)$ for $\alpha=2$, $\delta=0.64$, $\sigma^2=0.2$,  $p_{X_0}(\{+1\}) = p_{X_0}(\{-1\}) = 0.064$ and
$p_{X_0}(\{0\}) = 0.872$. }
  \label{fig:tau2->F(tau2)}
\end{figure}

\begin{figure}
\centering
  \includegraphics[width=4.7in]{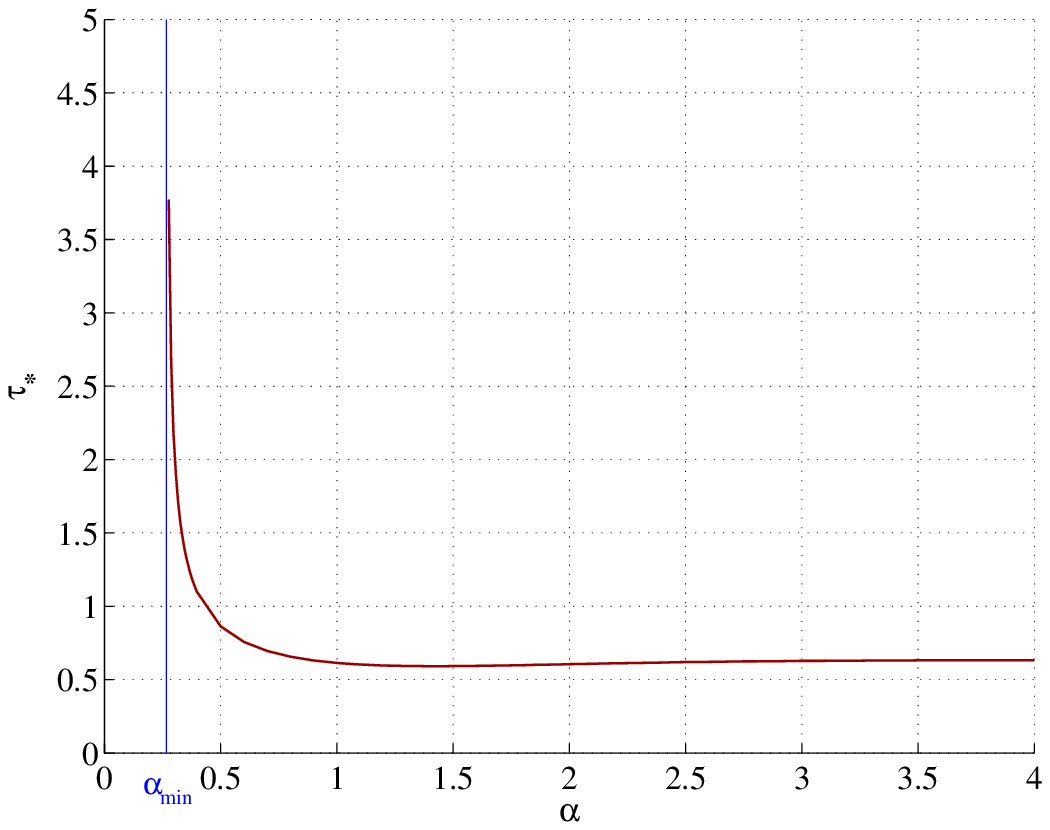}
  \caption{Mapping $\alpha\mapsto\tau_*(\alpha)$ for the same parameters $\delta$, $\sigma^2$ and distribution $p_{X_0}$ as in Figure \ref{fig:tau2->F(tau2)}.}
  \label{fig:alpha2taustar}
\end{figure}

\begin{figure}
\centering
  \includegraphics[width=4.7in]{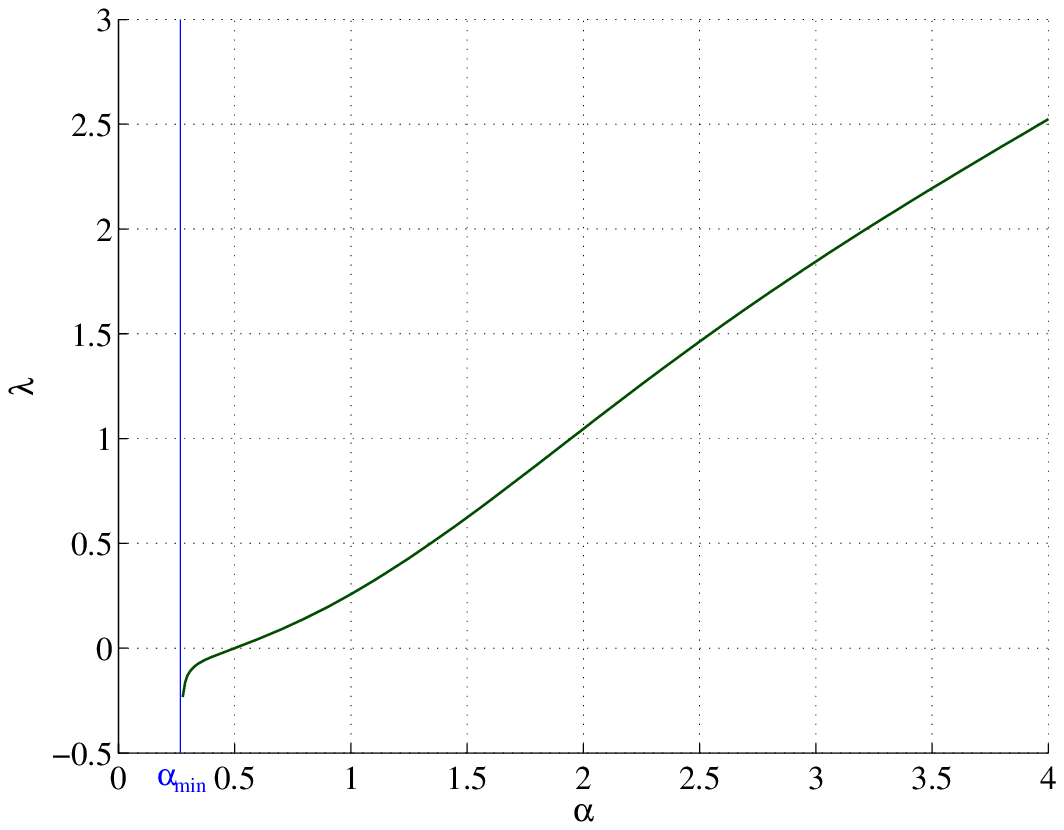}
  \caption{Mapping $\alpha\mapsto\lambda(\alpha)$ for the same parameters $\delta$, $\sigma^2$ and distribution $p_{X_0}$ as in Figure \ref{fig:tau2->F(tau2)}.}
  \label{fig:alpha2lam}
\end{figure}

%
\subsection{Related work}\label{sec:Related}

The LASSO was introduced in \cite{Tibs96,BP95}.
Several papers provide performance guarantees
for the LASSO or similar convex optimization methods
\cite{CandesStable,Dantzig}, by proving upper bounds on the
resulting mean squared error. These works assume an
appropriate `isometry' condition to hold for $A$.
While such condition hold with high probability for some
random matrices, it is often difficult to verify them explicitly.
Further, it is  only applicable to very sparse vectors
$x_0$. These restrictions are intrinsic to the worst-case
point of view developed in \cite{CandesStable,Dantzig}.

Guarantees have been proved for correct support recovery in
\cite{Zhao}, under an appropriate `incoherence' assumption on $A$.
While support recovery is an interesting conceptualization for
some applications (e.g. model selection), the metric considered
in the present paper (mean squared error) provides complementary
information and is quite standard
in many different fields.

Closer to the spirit of this paper \cite{Goyal}
derived expressions for the mean squared error under
the same model considered here. Similar results
were presented recently in \cite{KabashimaTanaka,BaronGuoShamai}.
These papers argue that a sharp asymptotic characterization
of the LASSO risk can provide valuable guidance in practical
applications. For instance, it can be used to evaluate competing
optimization methods on large scale applications, or to
tune the regularization parameter $\lambda$.

Unfortunately, these results were non-rigorous and were obtained
through the famously powerful `replica method' from statistical physics
\cite{MezardMontanari}.

Let us emphasize that the present paper offers two advantages
over these recent
developments: $(i)$ It is completely \emph{rigorous}, thus putting
on a firmer basis this line of research;
$(ii)$ It is \emph{algorithmic} in that
the LASSO mean squared error is shown to be equivalent to the one
achieved by a low-complexity message passing algorithm.
%
%
\section{Numerical illustrations}\label{sec:simulations}

Theorem \ref{thm:Risk} assumes that the entries of matrix $A$ have iid gaussian distribution.
We expect however the mean squared error
prediction  to be robust and hold for much
larger family of matrices.  Rigorous evidence in this direction is presented in \cite{KM-2010} where
the normalized cost $\cost(\hx)/N$ is shown to have a limit as
$N\to\infty$ which is universal with respect to
random matrices $A$ with iid entries.
(More precisely, it is universal provided $\E\{A_{ij}\}=0$,
$\E\{A_{ij}^2\}=1/n$ and $\E\{A_{ij}^6\}\le C/n^{3}$ for some uniform
constant $C$.)

Further, our result is asymptotic,
while and one might wonder how accurate it is for
instances of moderate dimensions.

Numerical simulations were carried out in \cite{NSPT,OurNips} and
suggest that the result is robust and relevant already for
$N$ of the order of a few hundreds.
As an illustration, we present in Figures \ref{fig:geneExpression}-\ref{fig:PM1Matrices} the outcome of such simulations
for  four types of real data and random matrices.
We generated the signal vector randomly with entries in $\{+1,0,-1\}$
and $\prob(x_{0,i}=+1) = \prob(x_{0,i}=-1) = 0.064$. The noise vector
$w$ was generated by using i.i.d. $\normal(0,0.2)$ entries.

We obtained the optimum estimator $\hx$ using \texttt{CVX}, a package for
specifying and solving convex programs \cite{CVX} and \texttt{OWLQN},
a package for solving large-scale versions of LASSO \cite{OWLQN}. We used
several values of $\lambda$ between $0$ and $2$ and $N$ equal to $200$,
$500$, $1000$, and $2000$. The aspect ratio of matrices was fixed in
all cases to $\delta=0.64$.
For each case, the point $(\lambda,\MSE)$ was plotted and the results
are shown in the figures. Continuous lines corresponds to the
asymptotic prediction by Corollary \ref{coro:lasso-risk}, namely $\delta(\tau_*^2-\sigma^2)$.

The agreement is remarkably good already for $N, n$ of
the order of a few hundreds, and deviations are consistent with statistical
fluctuations.

The four figures correspond to measurement matrices $A$:
\begin{itemize}
\item Figure \ref{fig:geneExpression}: Data consist of $2253$ measurements of expression level of $7077$ genes.
From this matrix we took sub-matrices $A$ of aspect ratio $\delta$ for each $N$. The entries were continuous variables.
We standardized all columns of $A$ to have mean 0 and variance 1.
\item Figure \ref{fig:Medical}: From a data set of $1932$ patient records we extracted $4833$ binary features describing demographic information, medical history, lab results, medications etc. The $0$-$1$ matrix was sparse
(with only $3.1\%$ non-zero entries). Similar to $(i)$, for each $N$, the sub-matrices $A$ with aspect ratio $\delta$ were selected and standardized.
\item Figure \ref{fig:GaussianMatrices}: Random gaussian matrices with aspect ratio $\delta$ and iid
$\normal(0,1/n)$ entries (as in Theorem \ref{thm:Risk});
\item Figure \ref{fig:PM1Matrices}: Random $\pm1$ matrices with aspect ratio $\delta$. Each entry is
independently equal to $+1/\sqrt{n}$ or $-1/\sqrt{n}$ with equal probability.
\end{itemize}
Notice the behavior appears to be essentially indistinguishable.
Also the asymptotic prediction has
a minimum as a function of $\lambda$. The location of this minimum
 can be used to select the regularization parameter.
Further empirical analysis is presented in \cite{OurUnpub}.

\begin{figure}
\centering
  \includegraphics[width=4.7in]{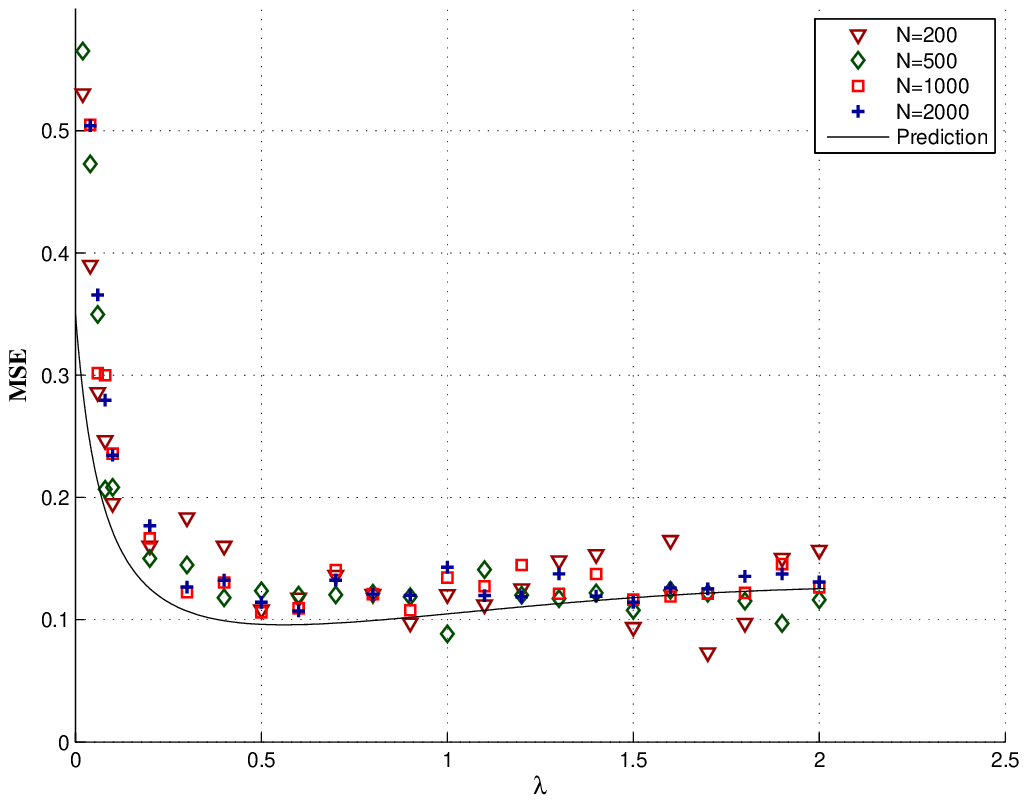}
  \caption{Mean squared error (MSE) as a function of the regularization parameter $\lambda$ compared to the asymptotic prediction for $\delta=.64$ and $\sigma^2=.2$.
Here the measurement matrix $A$ is a real valued (standardized) matrix of gene expression data. Each point in these plots is generated by finding the LASSO predictor $\hx$ using a measurement vector $y=Ax_0+w$ for an independent signal vector $x_0$ and an independent noise vector $w$.}
\label{fig:geneExpression}
\end{figure}
\begin{figure}
\centering
  \includegraphics[width=4.7in]{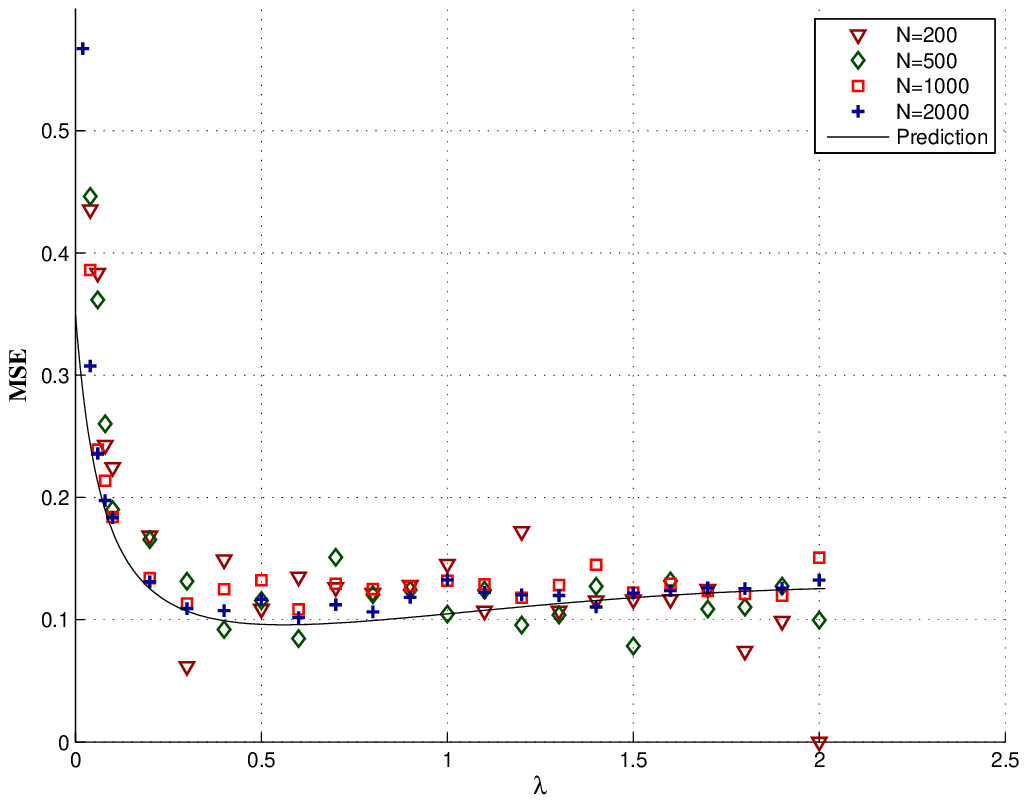}\\
  \caption{As in Figure \ref{fig:geneExpression}, but the measurement matrix $A$ is a (standardized) 0-1 feature matrix of hospital records.}\label{fig:Medical}
\end{figure}
\begin{figure}
\centering
  \includegraphics[width=4.7in]{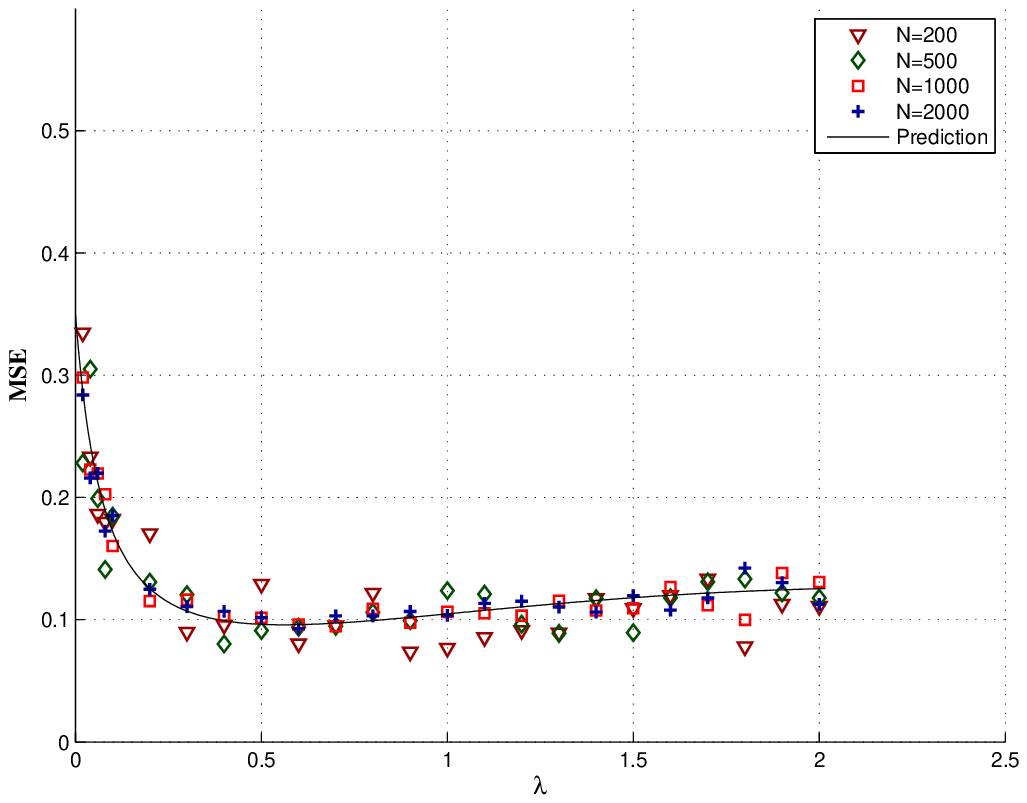}
  \caption{As in Figure \ref{fig:geneExpression}, but the measurement matrix $A$ has iid  $\normal(0,1/n)$ entries. Additionally, each point in this plot uses an independent matrix $A$.}\label{fig:GaussianMatrices}
\end{figure}

\begin{figure}
\centering
  \includegraphics[width=4.7in]{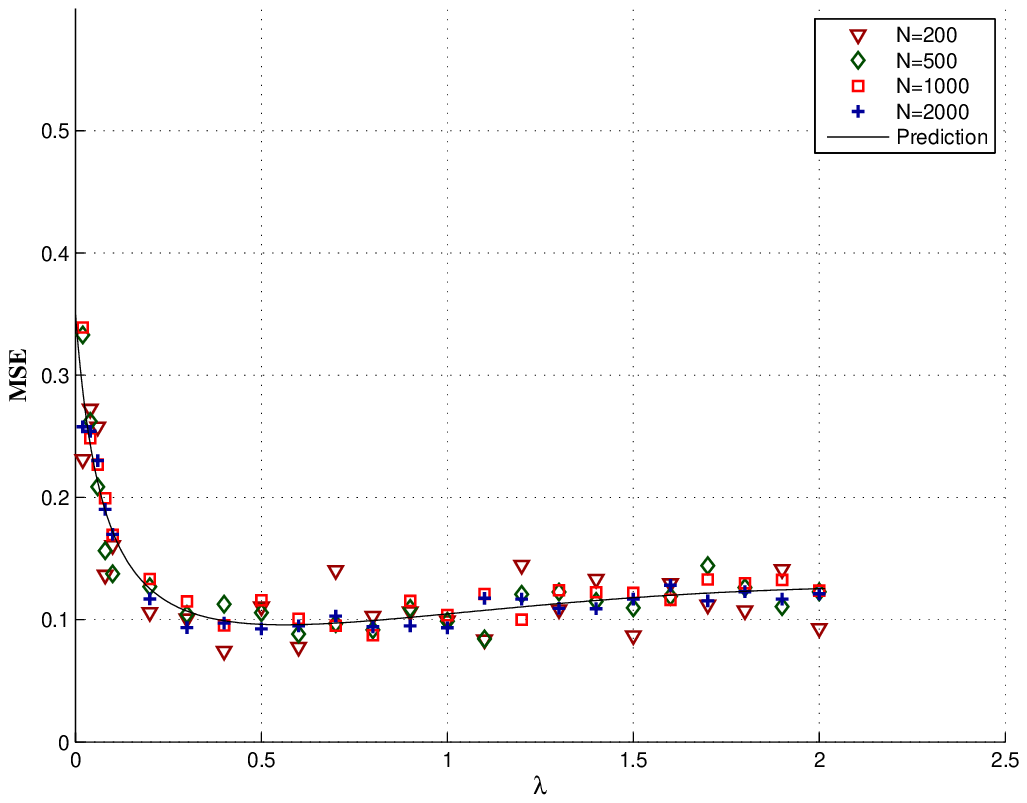}
 \caption{As in Figure \ref{fig:geneExpression}, but  the
measurement matrix $A$ has iid entries that are equal to $\pm1/\sqrt{n}$
with equal probabilities. Similar to Figure \ref{fig:GaussianMatrices}, each point in this plot uses an independent matrix $A$.}\label{fig:PM1Matrices}
\end{figure}
%
%
\section{A structural property
and proof of the main results}
\label{sec:MainProof}

The rest of the paper is devoted to the proof of Theorem \ref{thm:FiniteTime}. Section \ref{sec:Structural} proves a structural property that is the key tool in this proof. Section \ref{sec:ProofFiniteTime}
uses this property together with a few lemmas to prove
Theorem \ref{thm:FiniteTime}

The proof of Theorem \ref{thm:Risk} follows immediately from Theorem \ref{thm:FiniteTime}.
\begin{proof}[Proof of Theorem \ref{thm:Risk}]
For any $t\ge 0$, we have, by the pseudo-Lipschitz property of $\psi$,
\begin{align*}
\left|\frac{1}{N}\sum_{i=1}^N\psi
\big(x_{i}^{t+1},x_{0,i}\big)- \frac{1}{N}\sum_{i=1}^N\psi
\big(\hx_{i},x_{0,i}\big)\right|&\le
\frac{L}{N}\, \sum_{i=1}^N |x^{t+1}_i-\hx_i|
\big(1+2|x_{0,i}|+|x^{t+1}_i|+|\hx_i|\big)\\
&\le\frac{L}{N}\, \|x^{t+1}-\hx\|_2\,\sqrt{\sum_{i=1}^N\big(1+2|x_{0,i}|+|x^{t+1}_i|+|\hx_i|\big)^2}\\
&\le L \frac{\|x^{t+1}-\hx\|_2}{\sqrt{N}}\,\sqrt{4+\frac{8\|x_0\|_2^2}{N}+\frac{4\|x^{t+1}\|_2^2}{N}+\frac{4\|\hx\|_2^2}{N}}\,,\\
\end{align*}
where the second inequality follows by Cauchy-Schwarz.
Next we  take the limit $N\to\infty$ followed by $t\to\infty$.
The first term vanishes by Theorem \ref{thm:FiniteTime}.
For the second term, note that $\|x_0\|_2^2/N$ remains bounded
since $(x_0,w,A)$ is a converging sequence. The two terms
$\|x^{t+1}\|_2^2/N$ and $\|\hx\|_2^2/N$ also remain bounded in this
limit because of state evolution (as proved in  Lemma \ref{lemma:norm2(xt)is_bounded} below).

We then obtain
\begin{align*}
\lim_{N\to\infty} \frac{1}{N}\sum_{i=1}^N\psi
\big(\hx_{i},x_{0,i}\big)=
\lim_{t\to\infty}\lim_{N\to\infty}\frac{1}{N}\sum_{i=1}^N\psi
\big(x_{i}^{t+1},x_{0,i}\big) = \E\Big\{\psi\big(\eta(X_0+\tau_* Z;\theta_*),
X_0\big)\Big\}\, ,
\end{align*}
where we used  Theorem
\ref{prop:state-evolution} and Proposition \ref{propo:UniqFP}.
\end{proof}

%
%
\subsection{Some notations}

Before continuing, we introduce some useful notations.
For any non-empty subset $S$ of $[m]$ and any $k\times m$ matrix $M$
we refer by $M_S$ to the $k$ by $|S|$ sub-matrix of $M$ that contains only the
columns of $M$ corresponding to $S$. The same notation is used
for vectors $v\in\reals^m$: $v_S$ is the vector $(v_i:\, i\in S)$. For any vector $v\in\reals^m$ we denote support of $v$ by
\[\supp(v)\equiv\{i\,|\,v_i\neq 0\}\,.\]
We will also use the following scalar product for $u,v\in\reals^m$:
\begin{eqnarray}
\<u,v\> \equiv \frac{1}{m}\,\sum_{i=1}^m u_i\, v_i\, .
\end{eqnarray}
For a matrix $M$ we denote its minimum and maximum singular values by $\sigma_{\rm min}(M)$, $\sigma_{\rm max}(M)$ respectively. We also denote the minimum non-zero singular value of $M$ by $\hsigmamin(M)$.

The subgradient of a convex function $f:\reals^m\to\reals$
at point $x\in\reals^m$ is denoted by $\partial f(x)$. In particular,
remember that the subgradient of the $\ell_1$ norm,
$x\mapsto \|x\|_1$ is given by
\begin{eqnarray}
\partial\|x\|_1 =\big\{v\in\reals^m\mbox{ such that } |v_i|\le 1\,\forall i
\mbox{ and } x_i\neq 0\Rightarrow v_i = \sign(x_i)\big\}\, .
\end{eqnarray}

We will generally be interested in sequences of events $\{\cE_N\}$ indexed
by the problem dimensions $N$. It is understood throughout that the
underlying probability space is the one generated by the random
matrices $A(N)$, which we take to be independent across different
$N$. We say that such a sequence of events holds \emph{eventually
  almost surely} (as $N\to\infty$) if\footnote{Formally, if
$\prob(\cup_{\overline{N}\ge 1} \cap_{N\ge \overline{N}}\cE_N) =1$.} there exists a random variable
$N_0$ such that: $(i)$ $N_0$ is almost surely finite; $(ii)$ The events
$\cE_N$ hold for all $N\ge N_0$.
%
%
\subsection{A structural property of the LASSO cost function}
\label{sec:Structural}

One main challenge in the proof of Theorem \ref{thm:Risk}
lies in the fact that
the function $x\mapsto \cost_{A,y}(x)$ is not --in general--
strictly convex. Hence there can be, in principle, vectors $x$ of cost very
close to the optimum and nevertheless far from the optimum.

The following Lemma provides conditions under which this does not happen.
\begin{lemma}\label{lemma:smallcost2smallmse}
There exists a function $\xi(\eps,c_1,\dots,c_5)$ such that
the following happens.

If $x$, $\dx\in\reals^N$ satisfy the following conditions
\begin{enumerate}
\item $\|\dx\|_2\le c_1\sqrt{N}$;\label{H:Bound}
\item $\cost(x+\dx)\le \cost(x)$;\label{H:Cost}
\item There exists $\subg(\cost,x)\in\partial \cost(x)$ with $\|\subg(\cost,x)\|_2\le \sqrt{N}\, \eps$;
\label{H:Grad}
\item Let $v\equiv (1/\lambda)[A^*(y-Ax)+\subg(\cost,x)]\in \partial \|x\|_1$,
and $S(c_2)\equiv\{i\in [N]: \; |v_i|\ge 1-c_2\}$. Then,
for any $S'\subseteq [N]$, $|S'|\le c_3N$, we have
$\sigma_{\rm min}(A_{S(c_2)\cup S'})\ge c_4$;\label{H:Key}
\item The maximum singular value
of $A$ is bounded: $\sigma_{\rm max}(A)^2\le c_5$.
\label{H:SVD}
\end{enumerate}
Then $\|\dx\|_2\le \sqrt{N}\, \xi(\eps,c_1,\dots,c_5)$. Further
for any $c_1,\dots, c_5>0$, $\xi(\eps,c_1,\dots,c_5)\to 0$ as $\eps\to 0$.

Further, if $\ker(A)=\{0\}$, the same conclusion holds
under assumptions \ref{H:Bound}, \ref{H:Cost}, \ref{H:Grad}, \ref{H:SVD}.
\end{lemma}
\begin{proof}
Throughout the proof we denote $\xi_1, \xi_2,\dots$
functions of the constants $c_1,\dots,c_5>0$ and of $\eps$ such that
$\xi_i(\eps)\to 0$ as $\eps\to 0$ (we shall omit the dependence of
$\xi_i$ on $\eps$).

Let $S = \supp(x)\subseteq [N]$. We have
\begin{eqnarray*}
0 & \stackrel{(a)}{\ge} & \Big(\frac{ \cost(x+\dx)-\cost(x)}{N}\Big)\\
& \stackrel{(b)}{=} & \lambda\Big(\frac{\|x_S+\dx_S\|_1-\|x_S\|_1}{N}\Big)+\frac{\lambda\|\dx_{\Sco}\|_1+
\frac{1}{2}\|y-Ax-A\dx\|^2_2-\frac{1}{2}\|y-Ax\|^2_2}{N}\\
& \stackrel{(c)}{=} & \lambda\Big(\frac{\|x_S+\dx_S\|_1-\|x_S\|_1}{N}-\<\sign(x_S),\dx_S\>\Big)+
\lambda\Big(\frac{\|\dx_{\Sco}\|_1}{N}-\<v_{\Sco},\dx_{\Sco}\>\Big)
+\lambda\<v,\dx\>
-\<y-Ax,A\dx\>+\frac{\|A\dx\|_2^2}{2N}\\
& \stackrel{(d)}{=} & \lambda\Big(\frac{\|x_S+\dx_S\|_1-\|x_S\|_1}{N}-\<\sign(x_S),\dx_S\>\Big)+
\lambda\Big(\frac{\|\dx_{\Sco}\|_1}{N}-\<v_{\Sco},\dx_{\Sco}\>\Big)+
\<\subg(\cost,x),\dx\>+\frac{\|A\dx\|_2^2}{2N}\, ,
\end{eqnarray*}
where $(a)$ follows from hypothesis (\ref{H:Cost}),
$(c)$ from the fact that $v_S=\sign(x_S)$ since $v\in\partial\|x\|_1$ which gives
\[
\<\sign(x_S),r_S\>+\<v_{\Sco},r_{\Sco}\>=\<v_S,r_S\>+\<v_{\Sco},r_{\Sco}\>=\<v,r\>\,,
\]
and $(d)$ follows from the definition of $(v)$.

Using hypothesis (\ref{H:Bound}) and  (\ref{H:Grad}), we get by Cauchy-Schwarz
\begin{eqnarray}
 \lambda\Big(\frac{\|x_S+\dx_S\|_1-\|x_S\|_1}{N}-\<\sign(x_S),\dx_S\>\Big)+
\lambda\Big(\frac{\|\dx_{\Sco}\|_1}{N}-\<v_{\Sco},\dx_{\Sco}\>\Big)+\frac{\|A\dx\|_2^2}{2N}\le c_1\eps
\, .
\label{eq:three-non-neg}
\end{eqnarray}
Each of the three terms on the left-hand side is non-negative. The third one is trivial. The first one is non-negative since
\begin{eqnarray*}
\frac{\sum_{i\in S}\Big\{(x_i+r_i)\sign(x_i+r_i)-x_i\sign(x_i)-r_i\sign(x_i)\Big\}}{N}&=&\frac{\sum_{i\in S}(x_i+r_i)\big[\sign(x_i+r_i)-\sign(x_i)\big]}{N}\,,
\end{eqnarray*}
and each $(x_i+r_i)\left[\sign(x_i+r_i)-\sign(x_i)\right]$ is either equal to $0$ (when $\sign(x_i)=\sign(x_i+r_i)$) or equal to $2|x_i+r_i|$ otherwise. The second term in \eqref{eq:three-non-neg} is also non-negative since $|r_i|-v_ir_i = |r_i|[1-v_i\sign(r_i)]$ and $1\ge v_i\,\sign(r_i)$ since $|v_i|\le 1$ by definition of subgradient. Therefore,
\begin{eqnarray}
\frac{\|\dx_{\Sco}\|_1}{N}-\<v_{\Sco},\dx_{\Sco}\> &\le &\xi_1(\eps)\, ,\label{eq:Sco}\\
\|A\dx\|_2^2 & \le & N\xi_1(\eps)\, .\label{eq:AzBound}
\end{eqnarray}
Let $V_{\parallel}\subseteq\reals^{N}$ be the subspace of $\reals^N$
spanned by the right singular
vectors of $A$ with singular values $\sigma_i \le c_4/2$ (including
--eventually-- the null space of $A$), and denote by $V_{\perp}$ the
orthogonal complement of $V_{\parallel}$. Hence  $V_{\perp}$ is
spanned by right singular vectors of $A$ with singular value
$\sigma_i > c_4/2$. Let $\projpar$ and $\projperp$ denote the
orthogonal projectors on $V_{\parallel}$ and $V_{\perp}$.
Write $\dx = \dx^{\perp}+\dxparal$, with $\dxparal = \projpar \dx\in V_{\parallel}$
and  $\dx^{\perp}=\projperp\dx \in V_{\perp}$.
Also, write $A = \Apar+\Aperp \equiv A\projpar+A\projperp$
(note that $\Apar$ and $\Aperp$ have orthogonal column spaces).

It follows from Eq.~(\ref{eq:AzBound})  that
\begin{align}
\|\Apar\dxparal\|_2^2 \le N\xi_1(\eps)\,
  ,\;\;\;\;\;\;\|\Aperp\dx^{\perp}\|_2^2 \le N\xi_1(\eps)\, .
\end{align}
Since $\|\Aperp\dx^{\perp}\|_2^2\ge (c_4^2/4)\|\dx^{\perp}\|^2_2$, we have
\begin{eqnarray}
\|\dx^{\perp}\|_2^{2}\le \frac{4N\xi_1(\eps)}{c_4^2}\, .
\end{eqnarray}
In the case $V_{\parallel} = \{0\}$, the proof is concluded.
In the case $V_{\parallel}\neq\{0\}$, we need
to prove an analogous bound for $\dxparal$.
From Eq.~(\ref{eq:Sco}) together with
$\|\dx^{\perp}_{\Sco}\|_1\le\sqrt{N}\|\dx^{\perp}_{\Sco}\|_2
\le\sqrt{N}\|\dx^{\perp}\|_2 \le (2N/c_4) \sqrt{\xi_1(\eps)}$, we get
\begin{align}
&\frac{\|\dxparal_{\Sco}\|_1}{N}-\<v_{\Sco},\dxparal_{\Sco}\> \le\xi_2(\eps)\, .\label{eq:Sco2}\\
&\Aperp\dxparal= 0\, ,\label{eq:AperpX}
\end{align}
Where \eqref{eq:AperpX} follows immediately from definition of $\Aperp$ and $\dxparal$.
Now, notice that $\Sco(c_2)\subseteq\Sco$.
From Eq.~(\ref{eq:Sco2}) and definition of $S(c_2)$ it follows that
\begin{align}
\|\dxparal_{\Sco(c_2)}\|_1&\le \frac{\|\dxparal_{\Sco(c_2)}\|_1-N\<v_{\Sco(c_2)},\dxparal_{\Sco(c_2)}\>}{c_2}\label{eq:Sbar2Sbarc1}\\
&\le  Nc_2^{-1}\xi_2(\eps)\, \label{eq:Sbar2Sbarc2}.
\end{align}
In particular, inequality \eqref{eq:Sbar2Sbarc2} relies on the fact that the right hand side of \eqref{eq:Sbar2Sbarc1} can be written as $\sum_{i\in\Sco(c_2)}|\dxparal_i|(1-v_i\sign(\dxparal_i)|$ where each summand is non-negative, therefore the summation increases by replacing $\sum_{i\in\Sco(c_2)}$ with $\sum_{i\in\Sco}$. Next, let us first consider the case $|\Sco(c_2)|\ge Nc_3/2$. Then partition $\Sco(c_2) = \cup_{\ell=1}^K S_\ell$,
where $(Nc_3/2)\le |S_{\ell}|\le Nc_3$, and for each
$i\in S_{\ell}$, $j\in S_{\ell+1}$, $|\dxparal_i|\ge |\dxparal_j|$.
Also define $\Sco_+\equiv \cup_{\ell=2}^K S_{\ell}\subseteq \Sco(c_2)$.
Since, $|\dxparal_i|\le\|\dxparal_{S_{\ell-1}}\|_1/|S_{\ell-1}|$ holds for any $i\in S_{\ell}$,
we have
\begin{align*}
\|\dxparal_{\Sco_+}\|_2^2 &= \sum_{\ell=2}^K \|\dxparal_{S_{\ell}}\|_2^2\le
\sum_{\ell=2}^K |S_{\ell}|
\Big(\frac{\|\dxparal_{S_{\ell-1}}\|_1}{|S_{\ell-1}|}\Big)^2\\
&\le \frac{4}{Nc_3}\sum_{\ell=2}^{K} \|\dxparal_{S_{\ell-1}}\|_1^2\le
\frac{4}{Nc_3} \Big(\sum_{\ell=2}^{K}\|\dxparal_{S_{\ell-1}}\|_1\Big)^2\\
&\le
\frac{4}{Nc_3}\,\|\dxparal_{\Sco(c_2)}\|_1^2\le \frac{4\xi_2(\eps)^2}{c_2^2c_3}\, N
\equiv N\xi_3(\eps)\, .
\end{align*}

To conclude the proof, it is sufficient to prove an analogous bound for
$\|\dxparal_{S_+}\|_2^2$ with $S_+ = [N]\setminus \Sco_+= S(c_2)\cup S_1$.
Since $|S_1|\le Nc_3$, we have by hypothesis (\ref{H:Key})
that $\sigma_{\rm min}(A_{S_+})\ge c_4$.
By Eq.~(\ref{eq:AperpX}) we have $\Apar\dxparal=A\dxparal =
A_{S_+}\dxparal_{S_+}+A_{\Sco_+}\dxparal_{\Sco_+}$.
Therefore
\begin{align*}
c_4^2\|\dxparal_{S_+}\|_2^2\le \|A_{S_+}\dxparal_{S_+}\|^2_2= \|A_{\Sco_+}\dxparal_{\Sco_+}-\Apar\dxparal\|^2_2
\le 2c_5\|\dxparal_{\Sco_+}\|_2^2+2\frac{c_4^2}{4}\|\dxparal\|^2_2\, .
\end{align*}
In the last step we used triangular inequality together with the fact
that $\sigma_{\rm max}(A_{\Sco_+})^2\le c_5$ (by assumption
(\ref{H:SVD})) and $\sigma_{\rm max}(\Apar)\le c_4/2$ (by construction).
Using $\|\dxparal\|^2_2 =
\|\dxparal_{S_+}\|_2^2+\|\dxparal_{\Sco_+}\|_2^2$, we get
\begin{align*}
\frac{c_4^2}{2}\|\dxparal_{S_+}\|_2^2\le \Big(2
  c_5+\frac{c_4^2}{2}\Big)\|\dxparal_{\Sco_+}\|_2^2\le
\Big(2
  c_5+\frac{c_4^2}{2}\Big)N\xi_3(\eps)\, .
\end{align*}
This finishes the proof when $|\Sco(c_2)|\ge Nc_3/2$.
Note that if this assumption  does not hold then we can take $\Sco_+=\emptyset$ and
$S_+=[N]$. Hence, the result follows as a special case of above.
\end{proof}

%
%
\subsection{Proof of Theorem \ref{thm:FiniteTime}}
\label{sec:ProofFiniteTime}

The proof is based on a series of Lemmas that are used to check
the assumptions of Lemma \ref{lemma:smallcost2smallmse}

The first one
is  an upper bound on the $\ell_2$--norm of $\AMP$ estimates, and of
the $\LASSO$ estimate. Its proof is deferred to Section
\ref{sec:ProofNormBound}.
\begin{lemma}\label{lemma:norm2(xt)is_bounded}
Under the conditions of Theorem \ref{thm:Risk},
assume $\lambda>0$ and $\alpha = \alpha(\lambda)$.
Denote by $\hx(\lambda;N)$ the $\LASSO$ estimator and
by $\{x^t(N)\}$ the sequence of $\AMP$ estimates.
Then there is a constant $\finite$ such that
for all $t\ge 0$, almost surely
\begin{align}
\lim_{t\to\infty}\lim_{N\to\infty}\< x^t(N),x^t(N)\>&<\finite,\label{eq:Bound_xt}\\
\lim_{N\to\infty}\< \hx(\lambda;N),\hx(\lambda;N)\>&<\finite.
\end{align}
\end{lemma}

The second Lemma implies  that the estimates of $\AMP$ are approximate minima,
in the sense that the cost function $\cost$ admits a small subgradient at
$x^t$, when $t$ is large. The proof is deferred to Section
\ref{sec:small-subgradient}.
\begin{lemma}\label{lemma:small-subgradient} Under the conditions of Theorem \ref{thm:Risk}, for all $t$ there exists a subgradient $\subg(\cost,x^t)$ of $\cost$ at point $x^t$ such that almost surely,
\begin{align}
\lim_{t\to\infty}\lim_{N\to\infty}\frac{1}{N}\|\subg(\cost,x^t)\|^2=0.
\end{align}
\end{lemma}

The next lemma implies that sub-matrices of $A$ constructed
using the first $t$ iterations of the $\AMP$ algorithm
are non-singular (more precisely, have singular values bounded away
from $0$). The proof can be found in Section \ref{sec:ProofMinS}.
\begin{lemma}\label{lemma:MinS}
Let $S\subseteq [N]$ be measurable on the $\sigma$-algebra $\sigal{S}_t$
generated by $\{z^0,\dots, z^{t-1}\}$ and
$\{x^0+A^*z^0,\dots,x^{t-1}+A^*z^{t-1}\}$
and assume $|S|\le N(\delta-c)$ for some $c>0$. Then there
exists $a_1=a_1(c)>0$ (independent of $t$) and
$a_2=a_2(c,t)>0$ (depending on $t$ and $c$)
such that
\begin{eqnarray}
\min_{S'}\big\{\sigma_{\rm min}(A_{S\cup S'})\, :\;\; S'\subseteq [N]\, ,
\;|S'|\le a_1N\big\}
\ge a_2\,,
\end{eqnarray}
eventually almost surely as $N\to\infty$.
\end{lemma}

We will apply this lemma to a specific choice of the set $S$.
Namely, defining
\begin{eqnarray}
v^t\equiv\f{1}{\theta_{t-1}}(x^{t-1}+A^*z^{t-1}-x^t)\, ,\label{eq:vtDef}
\end{eqnarray}
we will then consider the set
\begin{eqnarray}
S_t(\gamma) \equiv \big\{\,i\in [N]\, :\; |v^t_{i}|\ge 1-\gamma\,\big\}\, ,
\label{eq:StDef}
\end{eqnarray}
for $\gamma\in (0,1)$. Our last lemma shows that this sequence of
sets $S_t(\gamma)$ `converges' in the following sense. The proof
can be found in Section \ref{sec:ConvergenceSupport}.
\begin{lemma}\label{lemma:ConvergenceSupport}
Fix $\gamma\in(0,1)$ and let the sequence $\{S_t(\gamma)\}_{t\ge 0}$
be defined as in Eq.~(\ref{eq:StDef}) above.
For any $\xi>0$ there exists $t_*=t_*(\xi,\gamma)<\infty$ such that,
for all $t_2\ge t_1\ge t_*$ fixed, we have
\begin{eqnarray}
|S_{t_2}(\gamma)\setminus S_{t_1}(\gamma)|< N\xi\,,
\end{eqnarray}
eventually almost surely as $N\to\infty$.
\end{lemma}
The above two lemmas imply the following.
\begin{proposition}\label{propo:PSD}
There exist constants $\gamma_1\in(0,1)$, $\gamma_2$, $\gamma_3>0$ and $t_{\rm min}<\infty$ such that,
for any $t\ge t_{\rm min}$,
\begin{eqnarray}
\min_{S_1}\big\{\sigma_{\rm min}(A_{S_t(\gamma_1)\cup S'})\, :\;\; S'\subseteq [N]\, ,
\;|S'|\le \gamma_2 N\big\}
\ge \gamma_3\,,
\end{eqnarray}
eventually almost surely as $N\to\infty$.
\end{proposition}
\begin{proof}
First notice that, for any fixed $\gamma$,
the set $S_t(\gamma)$ is  measurable on $\sigal{S}_t$.
Indeed by Eq.~(\ref{eq:dmm}) $\sigal{S}_t$
contains $\{x^0,\dots, x^t\}$ as well,
and hence it contains  $v^t$ which is
a  linear combination of $x^{t-1}+A^*z^{t-1}$, $x^t$.
Finally $S_t(\gamma)$ is obviously a measurable function of $v^t$.

Using Lemma \ref{lem:elephant}(b) the empirical distribution of $(x_0-A^*z^{t-1}-x^{t-1},x_0)$
converges weakly to $(\tau_{t-1}Z,X_0)$ for $Z\sim \normal(0,1)$ independent of $X_0\sim p_{X_0}$.
(Following the notation of \cite{BM-MPCS-2010}, we let
$h^t=x_0-A^*z^{t-1}-x^{t-1}$.)
Therefore, for any constant $\gamma$ we have almost surely
\begin{align}
\lim_{N\to\infty}\f{|S_t(\gamma)|}{N}&= \lim_{N\to\infty}\f{1}{N}\sum_{i=1}^N \ind_{\left\{\f{1}{\theta_{t-1}}\big|x_i^{t-1}+[A^*z^{t-1}]_i-x_i^t\big|\geq 1-\gamma\right\}}\\
&= \lim_{N\to\infty}\f{1}{N}\sum_{i=1}^N \ind_{\left\{\f{1}{\theta_{t-1}}\big|x_{0,i}-h_i^t-\eta(x_{0,i}-h_i^t,\theta_{t-1})\big|\geq 1-\gamma\right\}}\\
&= \prob\left\{\f{1}{\theta_{t-1}}\big|X_0+\tau_{t-1}Z-\eta(X_0+\tau_{t-1}Z,\theta_{t-1})\big|\geq 1-\gamma\right\} \,.\label{eq:h-x_0-state-evolution}
\end{align}
The last equality follows from the weak convergence of the empirical
distribution of $\{(h_i,x_{0,i})\}_{i\in [N]}$
(from Lemma \ref{lem:elephant}(b), which takes
the same form as Theorem \ref{thm:FiniteTime}),
together with the absolute continuity of the
distribution of $|X_0+\tau_{t-1}Z-\eta(X_0+\tau_{t-1}Z,\theta_{t-1})|$.

Now, combining
\[
\Big|X_0+\tau_{t-1}Z-\eta(X_0+\tau_{t-1}Z,\theta_{t-1})\Big|=
\left\{
\begin{array}{ll}
\theta_{t-1}&{\rm when}~~~|X_0+\tau_{t-1}Z|\geq \theta_{t-1}\,,\\
|X_0+\tau_{t-1}Z|&{\rm otherwise}\,,
\end{array}
\right.
\]
and Eq. \eqref{eq:h-x_0-state-evolution} we obtain almost surely
\begin{align}
\lim_{N\to\infty}\f{|S_t(\gamma)|}{N}&=\E\left\{\eta'(X_0+\tau_{t-1}Z,\theta_{t-1})\right\}+\prob\Big\{
(1-\gamma)\leq \f{1}{\theta_{t-1}}|X_0+\tau_{t-1}Z|\leq 1\Big\}.
\end{align}
It is easy to see that the second term $\prob\left\{1-\gamma\leq (1/\theta_{t-1})|X+\tau_{t-1}Z|\leq 1\right\}$ converges to $0$ as $\gamma\to 0$.
On the other hand, using Eq. \eqref{eq:calibration} and
the fact that $\lambda(\alpha)>0$ the first term will be strictly smaller than
$\delta$ for large enough $t$.
Hence, we can choose constants $\gamma_1\in(0,1)$ and $c>0$ such that
\begin{eqnarray}
|S_t(\gamma_1)|< N(\delta - c)\, .
\end{eqnarray}
eventually almost surely as $N\to\infty$,
for all fixed $t$ larger than some $t_{{\rm min},1}(c)$.

For any $t\ge t_{{\rm min},1}(c)$ we can apply Lemma \ref{lemma:MinS}
for some $a_1(c)$, $a_2(c,t)>0$.
Fix $c>0$ and let $a_1=a_1(c)$ be fixed as well.
Let $t_{\rm min} =\max(t_{{\rm min},1},
t_{*}(a_1/2,\gamma_1))$ (with $t_*(\,\cdot\,)$ defined as per
Lemma \ref{lemma:ConvergenceSupport}).
Take $a_2=a_2(c,t_{\rm min})$. Obviously $t\mapsto a_2(c,t)$
is non-increasing.
Then we have,
by Lemma \ref{lemma:MinS}
\begin{eqnarray}
\min\big\{\sigma_{\rm min}(A_{S_{t_{\rm min}}(\gamma_1)\cup S'})\, :\;\;
S'\subseteq [N]\, ,\;|S'|\le a_1N\big\}
\ge a_2\, ,
\end{eqnarray}
and by Lemma \ref{lemma:ConvergenceSupport}
\begin{eqnarray}
|S_{t}(\gamma_1)\setminus S_{t_{\rm min}}(\gamma_1)|\leq Na_1/2,
\end{eqnarray}
where both events hold eventually almost surely as $N\to\infty$.
The claim follows with $\gamma_2 = a_1(c)/2$ and $\gamma_3=a_2(c,t_{\rm min})$.
\end{proof}

We are now in position to prove Theorem \ref{thm:FiniteTime}.
\begin{proof}[Proof of Theorem \ref{thm:FiniteTime}]
We apply Lemma \ref{lemma:smallcost2smallmse}
to $x= x^t$, the $\AMP$ estimate and $r = \hx-x^t$ the distance
from the $\LASSO$ optimum. The thesis follows by checking
conditions 1--5. Namely we need to show that there
exists constants $c_1,\dots,c_5>0$ and, for each $\ve>0$
some $t=t(\ve)$ exists such that 1--5 hold eventually almost surely
as $N\to\infty$.

\vspace{0.2cm}

\emph{Condition 1} holds by Lemma \ref{lemma:norm2(xt)is_bounded}.

\vspace{0.2cm}

\emph{Condition 2} is immediate since $x+r = \hx$ minimizes $\cost(\,\cdot\,)$.

\vspace{0.2cm}

\emph{Condition 3} follows from Lemma \ref{lemma:small-subgradient}
with $\eps$ arbitrarily small for  $t$ large enough.

\vspace{0.2cm}

\emph{Condition 4.}
Notice that this condition only needs to be verified for $\delta<1$.

Take $v=v^t$ as defined in Eq.~(\ref{eq:vtDef}).
Using the definition (\ref{eq:dmm}), it is easy to check that
$|v_{i}^t|\le 1$ if $x_{i}^t = 0$ and $v_{i}^t = \sign(x^t_i)$ otherwise.
In other words $v^t\in\partial \|x\|_1$ as required.
Further by inspection of the proof of Lemma \ref{lemma:small-subgradient},
it follows that  $v^t= (1/\lambda)[A^*(y-Ax^t)+\subg(\cost,x^t)]$,
with $\subg(\cost,x^t)$ the subgradient bounded in that lemma
(cf. Eq.~(\ref{eq:s(xt)_defined})).
The condition then holds by Proposition \ref{propo:PSD}.

\vspace{0.2cm}

\emph{Condition 5} follows from standard limit theorems on the singular values
of Wishart matrices (cf. Theorem \ref{prop:marchenko-pastur}).
\end{proof}
%
%
\section{State evolution estimates}

This section contains a reminder of the state-evolution method
developed in \cite{BM-MPCS-2010}. For greater convenience of the
reader,  we also restate two lemmas from  \cite{BM-MPCS-2010} (namely,
 Lemmas \ref{lem:elephant} and \ref{lem:elephant}) in appendix \ref{sec:FromMPCS}. We will use these two Lemmas throughout our analysis.

We also state some extensions of those
results that will be proved in the appendices.
%
%
\subsection{State evolution}

$\AMP$, cf. Eq.~(\ref{eq:dmm}) is a special case of the general iterative procedure given by  Eq. (3.1) of \cite{BM-MPCS-2010}. This takes the general form
\begin{eqnarray}
h^{t+1}&=& A^*m^t-\xi_t\, q^t\, ,\;\;\;\;\;\;\;\;
m^t=g_t(b^t,w)\, ,\nonumber\\
b^t&=& A\,q^t-\lambda_t m^{t-1} \, ,\;\;\;\;\;\;\;
q^t=f_t(h^t,x_0)\, ,\label{eq:mpMain}
\end{eqnarray}
where $\xi_t = \< g'(b^t,w)\>$, $\lambda_t=\f{1}{\delta}
\< f'_t(h^{t},x_0)\>$ (both derivatives are with respect to the first
argument).

This reduction can be seen by defining
\begin{align}
h^{t+1} &= x_0-(A^*z^t+x^t)\, ,\label{eq:h-as-z-and-x}\\
q^{t} &= x^t-x_0\, ,\label{eq:q-as-x}\\
b^{t} &= w-z^t\, ,\label{eq:b-as-z}\\
m^{t} &= -z^t\, ,\label{eq:m-as-z}
\end{align}
where
\begin{eqnarray}
f_t(s,x_0)=\eta_{t-1}(x_0-s)-x_0\, ,\;\;\;\;\;\;\;\;\;\;
g_t(s, w)=s-w\, ,
\end{eqnarray}
and the initial condition is $q^0=-x_0$.

Regarding $h^{t},b^t$ as column vectors, the equations for $b^0,\ldots,b^{t-1}$ and $h^1,\ldots,h^{t}$ can be written in matrix form as:
\begin{align}
\underbrace{\left[h^1+\xi_0q^0|h^2+\xi_1q^1|\cdots|h^t+\xi_{t-1}q^{t-1}\right]}_{X_t}&=A^*\underbrace{[m^0|\ldots|m^{t-1}]}_{M_t}\,,\label{eq:X=A*M}\\
\underbrace{\left[b^0|b^1+\lambda_1m^0|\cdots|b^{t-1}+\lambda_{t-1}m^{t-2}\right]}_{Y_t}&=A\underbrace{[q^0|\ldots|q^{t-1}]}_{Q_t}\,.\label{eq:Y=AQ}
\end{align}
or in short $Y_t=AQ_t$ and $X_t=A^*M_t$.

Following \cite{BM-MPCS-2010}, we
define $\sigal{S}_{t}$ as the $\sigma$-algebra generated
by $b^0,\ldots,b^{t-1}$, $m^0,\ldots,m^{t-1}$, $h^1,\ldots,h^{t}$,
and $q^0,\ldots,q^{t}$.
The conditional distribution of the random matrix $A$
given the $\sigma$-algebra ${\sigal{S}_{t}}$, is given by
\begin{align}
A|_{\sigal{S}_{t}}&\deq E_{t} + \cP_{t}(\tA) \label{eq:Conditional_A}.
\end{align}
Here $\tA\deq A$ is a random matrix independent of ${\sigal{S}_{t}}$,
and $E_{t}=\E(A|{\sigal{S}_{t}})$ is given by
\begin{align}
E_{t}&=Y_{t}(Q_{t}^*Q_{t})^{-1}Q_{t}^*+M_{t}(M_{t}^*M_{t})^{-1}X_{t}^*-M_{t}(M_{t}^*M_{t})^{-1}M_{t}^*Y_{t}(Q_{t}^*Q_{t})^{-1}Q_{t}^*\, .\label{eq:Et}
\end{align}
Further, $\cP_{t}$ is the orthogonal projector onto subspace $\mathrm{V}_{t}=\{A|AQ_{t}=0,A^*M_{t}=0\}$, defined by
\[
\cP_{t}(\tA) =P_{M_{t}}^\perp {\tA} P_{Q_{t}}^\perp.
\]
Here $P_{M_{t}}^\perp = I-P_{M_{t}}$, $P_{Q_{t}}^\perp=I-P_{Q_{t}}$, and $P_{Q_{t}}$, $P_{M_{t}}$ are orthogonal projector onto column spaces of $Q_{t}$ and $M_{t}$ respectively.

Before proceeding, it is convenient to introduce  the notation
\[
\onsager_t\equiv\f{1}{\delta}\<\eta'(A^*z^{t-1}+x^{t-1};\theta_{t-1})\>
\]
to denote the coefficient of $z^{t-1}$ in Eq.~\eqref{eq:dmm}.
Using $h^t= x_0- A^*z^{t-1}-x^{t-1}$ and Lemma \ref{lem:elephant}(b)
(proved in  \cite{BM-MPCS-2010}) we get, almost surely,
\begin{eqnarray}
\label{eq:LimOmega_t}
\lim_{N\to\infty} \omega_t = \omega^{\infty}_t\equiv
\f{1}{\delta} \E\big[
\eta'(X_0+\tau_{t-1}Z;\theta_{t-1})\big]\, .
\end{eqnarray}
%
Notice that the function $\eta'(\,\cdot\,;\theta_{t-1})$ is discontinuous
and therefore Lemma \ref{lem:elephant}(b) does not apply
immediately. On the other hand, this implies that the empirical distribution of
$\{(A^*z^{t-1}_i+x^{t-1}_i,x_{0,i})\}_{1\le i\le N}$ converges weakly to
the distribution of $(X_0+\tau_{t-1}Z,X_0)$. The claim follows from
the fact that $X_0+\tau_{t-1}Z$ has a density, together with the
standard properties of weak convergence.

%
%
\subsection{Some consequences and generalizations}

We begin with a simple calculation, that will be useful.
\begin{lemma}\label{lemma:NormZ}
If $\{z^t\}_{t\ge 0}$ are the $\AMP$ residuals, then, almost surely,
\begin{align}
\lim_{n\to\infty} \f{1}{n}\, \|z^t\|^2  =\tau_t^2\, .
\end{align}
\end{lemma}
\begin{proof}
Using representation \eqref{eq:m-as-z} and Lemma \ref{lem:elephant}(b)(c), we get
\begin{align*}
\lim_{n\to\infty} \f{1}{n}\, \|z^t\|^2
\asequal  \lim_{n\to\infty} \f{1}{n}\, \|m^t\|^2\asequal  \lim_{N\to\infty} \f{1}{N}\, \|h^{t+1}\|^2= \tau_t^2\, .
\end{align*}
\end{proof}

Next, we need to generalize state evolution to compute
large system limits for functions of
$x^t$, $x^s$, with $t\neq s$. To this  purpose,
we define the covariances $\{\covz_{s,t}\}_{s,t\ge 0}$ recursively by
\begin{eqnarray}
\covz_{s+1,t+1} = \sigma^2+\frac{1}{\delta}\,\E\Big\{
[\eta(X_0+Z_s;\theta_s)-X_0]\, [\eta(X_0+Z_t;\theta_t)-X_0]\Big\}\, ,
\label{eq:2-times-SE}
\end{eqnarray}
with $(Z_s,Z_t)$ jointly gaussian, independent from $X_0\sim p_{X_0}$
with zero mean and covariance given by $\E\{Z^2_s\} = \covz_{s,s}$,
$\E\{Z^2_t\} = \covz_{t,t}$, $\E\{Z_sZ_t\} = \covz_{s,t}$.
The boundary condition is fixed by letting $\covz_{0,0} = \sigma^2
+\E\{X_0^2\}/\delta$ and
\begin{eqnarray}
\covz_{0,t+1} = \sigma^2+\frac{1}{\delta}\,\E\Big\{
[\eta(X_0+Z_t;\theta_t)-X_0]\, (-X_0)\Big\}\, ,\label{eq:Initial-2times-SE}
\end{eqnarray}
with $Z_{t}\sim\normal(0,\covz_{t,t})$ independent of $X_0$.
This determines by the above recursion $\covz_{t,s}$ for all $t\ge 0$ and for all $s\ge 0$.

With these definition, we have the following generalization of
Theorem \ref{prop:state-evolution}.
\begin{theorem}\label{prop:state-evolution-2times}
Let $\{x_0(N), w(N), A(N)\}_{N\in\naturals}$ be a converging sequence of
instances with the entries of $A(N)$ iid normal with mean $0$ and variance
$1/n$ and let $\psi:\reals^3\to \reals$ be a pseudo-Lipschitz
function. Then, for all $s\ge0$ and $t\ge 0$ almost surely
\begin{eqnarray}\label{eq:state-evolution-2times}
\lim_{N\to\infty}\frac{1}{N}\sum_{i=1}^N\psi
\big(x_{i}^{s}+(A^*z^s)_i,x_{i}^{t}+(A^*z^t)_i,x_{0,i}\big) =
\E\Big\{\psi\big(X_0+Z_s,
X_0+Z_t,X_0\big)\Big\}\, ,
\end{eqnarray}
where $(Z_s,Z_t)$ jointly gaussian, independent from $X_0\sim p_{X_0}$
with zero mean and covariance given by $\E\{Z^2_s\} = \covz_{s,s}$,
$\E\{Z^2_t\} = \covz_{t,t}$, $\E\{Z_sZ_t\} = \covz_{s,t}$.
\end{theorem}
Notice that the above implies in particular, for any
pseudo-Lipschitz function $\psi:\reals^3\to \reals$,
\begin{eqnarray}
\lim_{N\to\infty}\frac{1}{N}\sum_{i=1}^N\psi
\big(x_{i}^{s+1},x_{i}^{t+1},x_{0,i}\big) =
\E\Big\{\psi\big(\eta(X_0+Z_s;\theta_s),
\eta(X_0+Z_t;\theta_t),X_0\big)\Big\}\, .
\end{eqnarray}
Clearly this result reduces to Theorem \ref{prop:state-evolution}
in the case $s=t$ by noting that $\covz_{t,t}=\tau^2_t$.
The general proof can be found in Appendix \ref{app:state-evolution-2times}.

The following lemma implies that, asymptotically for large $N$,
the $\AMP$ estimates converge.
\begin{lemma} \label{lemma:Convergence}
Under the condition of Theorem \ref{thm:Risk}, the estimates $\{x^t\}_{t\ge 0}$ and residuals
$\{z^t\}_{t\ge 0}$ of $\AMP$ almost surely satisfy
\label{lemma:norm2(x(t)-x(t+1))is_small}
\begin{equation}
\lim_{t\to\infty}\lim_{N\to\infty}\f{1}{N}
\|x^t-x^{t-1}\|^2=0\,,~~~~~\lim_{t\to\infty}\lim_{N\to\infty}\f{1}{N}
\|z^t-z^{t-1}\|^2=0\,.
\end{equation}
\end{lemma}
The proof is deferred to Appendix \ref{app:Convergence}.

%
%
\section{Proofs of auxiliary lemmas}

\subsection{Proof of Lemma \ref{lemma:norm2(xt)is_bounded}}
\label{sec:ProofNormBound}

In order to bound the norm of $x^t$, we use state evolution,
Theorem \ref{prop:state-evolution}, for the function $\psi(a,b)=a^2$,
\[
\lim_{t\to\infty}\lim_{N\to\infty}\< x^t,x^t\>\asequal \E\left\{\eta(X_0+\tau_* Z;\theta_*)^2\right\}
\]
for $Z\sim\normal(0,1)$ and independent of $X_0\sim p_{X_0}$. The expectation on the right hand side is bounded and hence
$\lim_{t\to\infty}\lim_{N\to\infty}\< x^t,x^t\>$ is bounded.

For $\hx$, first note that
\begin{align}
\f{1}{N}\cost(\hx)\le\f{1}{N}\cost(0)&=
\f{1}{2N}\|y\|^2\nonumber\\
&=\f{1}{2N}\|Ax_0+w\|^2\nonumber\\
&\le
\frac{\|w\|^2+\smaxA^2\|x_0\|^2}{N}\le \finite_1.
\label{eq:BoundCost}
\end{align}
The last bound holds almost surely as $N\to\infty$,
using standard asymptotic estimate on the singular values of
random matrices (cf. Theorem \ref{prop:marchenko-pastur})
implying that $\smaxA$ has a bounded limit almost surely,
together with the fact that $(x_0,w, A)$ is a converging sequence.

Now, decompose $\hx$ as $\hx = \hx_{\pl}+\hx_{\perp}$
where $\hx_{\pl}\in\ker(A)$ and $\hx_{\perp}\in\ker(A)^\perp$
(the orthogonal complement of $\ker(A)$).
Since, $\hx_{\pl}$ belongs to the random subspace $\ker(A)$ with
dimension $N-n=N(1-\delta)$,
Kashin theorem (cf. Theorem \ref{thm:Kashin})
implies that there exists a positive constant $c_1=c_1(\delta)$ such that
\begin{align*}
\f{1}{N}\|\hx\|^2&= \f{1}{N}\|\hx_{\pl}\|^2+
\f{1}{N}\|\hx_{\perp}\|^2\\
&\le c_1
\left(\frac{\|\hx_{\pl}\|_1}{N}\right)^2+\frac{1}{N}
\, \|\hx_{\perp}\|^2\, .
\end{align*}
Hence, by using triangle inequality and Cauchy-Schwarz, we get
\begin{align*}
\f{1}{N}\|\hx\|^2
&\le 2c_1\left(\frac{\|\hx\|_1}{N}\right)^2
+2c_1\left(\frac{\|\hx_{\perp}\|_1}{N}\right)^2+\f{1}{N}
\|\hx_{\perp}\|^2\\
&\le 2c_1\left(\frac{\|\hx\|_1}{N}\right)^2+\f{2c_1+1}{N}
\|\hx_{\perp}\|^2\, .
\end{align*}
By definition of cost function we have $\|\hx\|_1\le
\lambda^{-1}\cost(\hx)$.
Further, limit theorems for the eigenvalues of
Wishart matrices (cf. Theorem \ref{prop:marchenko-pastur})
imply that there exists a constant $c = c(\delta)$
such that asymptotically almost surely
$\|\hx_{\perp}\|^2\le c\, \|A\hx_{\perp}\|^2$.
Therefore (denoting by $c_i:~i=2,3,4$ bounded constants), we have
\begin{align*}
\f{1}{N}\|\hx\|^2
&\le 2c_1\left(\frac{\|\hx\|_1}{N}\right)^2+\f{c_2}{N}
\|A\hx_{\perp}\|^2\\
&\le 2c_1\left(\frac{\|\hx\|_1}{N}\right)^2+\f{2c_2}{N}
\|y-A\hx_{\perp}\|^2+\f{2c_2}{N}\|y\|^2\\
&\le c_3\left(\frac{\cost(\hx)}{N}\right)^2+2c_2
\f{\cost(\hx)}{N}
+\f{2c_2}{N}\|Ax_0+w\|^2\, .
\end{align*}
The claim follows by using the Eq.~(\ref{eq:BoundCost}) to bound
$\cost(\hx)/N$ and using
$\|Ax_0+w\|^2\le\smaxA^2\|x_0\|^2+\|w\|^2\le 2N\finite_1$
to bound the last term.
\eprooft
%
%
\subsection{Proof of Lemma \ref{lemma:small-subgradient}}
\label{sec:small-subgradient}

First note that equation $x^t=\eta(A^*z^{t-1}+x^{t-1};\theta_{t-1})$ of $\AMP$
implies
\begin{eqnarray}
x_i^t+\theta_{t-1}\,\sign(x_i^t) = [A^*z^{t-1}]_i+x_i^{t-1},&&~~~\textrm{if } x_i^t\neq 0\,,\nonumber\\
&&\label{eq:x(t)=eta_interpreted}\\
\Big|[A^*z^{t-1}]_i+x_i^{t-1}\Big|\leq \theta_{t-1}, &&~~~\textrm{if }x_i^t=0\,. \nonumber
\end{eqnarray}
Therefore, the vector $\subg(\cost,x^t)\equiv\lambda\,s^t-A^*(y-Ax^t)$ where
\begin{equation}
\label{eq:s(xt)_defined}
s_i^t=
\left\{
\begin{array}{lll}
\sign(x_i^t)&&~~~\textrm{if } x_i^t\neq 0\, ,\\
&&\\
\frac{1}{\theta_{t-1}}\Big\{[A^*z^{t-1}]_i+x_i^{t-1}\Big\} &&~~~\textrm{otherwise,}
\end{array}\right.
\end{equation}
is a valid subgradient of $\cost$ at $x^t$.
On the other hand, $y-Ax^t=z^t-\onsager_tz^{t-1}$. We finally get
\begin{align*}
\subg(\cost,x^t)&=\f{1}{\theta_{t-1}}\left[\lambda\theta_{t-1}s^t-\theta_{t-1}A^*(z^t-\onsager_tz^{t-1})\right]\\
&=\f{1}{\theta_{t-1}}\left[\lambda\theta_{t-1}s^t-\theta_{t-1}(1-\onsager_t)A^*z^{t-1}\right]-A^*(z^t-z^{t-1})\\
&=\underbrace{\f{1}{\theta_{t-1}}\left[\lambda\theta_{t-1}s^t-\lambda A^*z^{t-1}\right]}_{(I)}-A^*(z^t-z^{t-1})+\f{[\lambda -\theta_{t-1}(1-\onsager_t) ]}{\theta_{t-1}}\,A^*z^{t-1}\,.
\end{align*}
It is straightforward to see from Eqs. \eqref{eq:x(t)=eta_interpreted} and \eqref{eq:s(xt)_defined} that $(I)=\lambda(x^{t-1}-x^t)$.
Hence,
\begin{align*}
\frac{1}{\sqrt{N}}\|\subg(\cost,x^t)\|&\le \frac{\la}{\theta_{t-1}\sqrt{N}}\|x^t-x^{t-1}\| + \frac{\smaxA}{\sqrt{N}}\|z^t-z^{t-1}\|+\f{|\lambda -\theta_{t-1}(1-\onsager_t)|}{\theta_{t-1}}\, \frac{1}{\sqrt{N}}\|z^{t-1}\|\,.
\end{align*}
By Lemma \ref{lemma:norm2(x(t)-x(t+1))is_small},
and the fact that $\smaxA$ is almost surely bounded as $N\to \infty$
(cf. Theorem \ref{prop:marchenko-pastur}),
we deduce that the two terms
$\la\|x^t-x^{t-1}\|/(\theta_{t-1}\sqrt{N})$
and $\smaxA\|z^t-z^{t-1}\|^2/\sqrt{N}$ converge
to $0$ when $N\to\infty$ and then $t\to\infty$. For the third term, using state evolution
(see Lemma \ref{lemma:NormZ}), we obtain
$\lim_{N\to\infty}\|z^{t-1}\|^2/N<\infty$.  Finally,
using the calibration relation Eq.~\eqref{eq:calibration}, we get
\begin{align*}
\lim_{t\to\infty}\lim_{N\to\infty}\left|\f{\lambda -\theta_{t-1}(1-\onsager_t)}{\theta_{t-1}}\right|&\asequal\f{1}{\theta_*}\left|\lambda -\theta_*(1-\frac{1}{\delta}
\E\left\{\eta'(X_0+\tau_* Z;\theta_*)\right\})\right|= 0\, ,
\end{align*}
which finishes the proof.
\eprooft
%
%
\subsection{Proof of Lemma \ref{lemma:MinS}}
\label{sec:ProofMinS}

The proof uses the representation (\ref{eq:Conditional_A}),
together with the expression (\ref{eq:Et}) for the conditional expectation.
Apart from the matrices $Y_t$, $Q_t$, $X_t$, $M_t$ introduced there,
we will also use
\begin{align*}
B_t \equiv \Big[b^0\Big|b^1\Big|\cdots\Big|b^{t-1}\Big]\,,
\;\;\;\;\;\;\;\;
H_t \equiv \Big[h^1\Big|h^2\Big|\cdots\Big|h^{t}\Big]\, .
\end{align*}
In this section, since $t$ is fixed, we will drop everywhere the
subscript $t$ from such matrices.

We state below a somewhat more convenient description.
\begin{lemma}\label{lemma:AvApprox}
For any $v\in\reals^N$, we have
\begin{eqnarray}
Av|_{\sigal{S}} \stackrel{d}{=} Y(Q^*Q)^{-1}Q^*P_Qv + M(M^*M)^{-1}X^*P_Q^{\perp}v +
P_M^{\perp}\tA P_Q^{\perp}v \, .
\end{eqnarray}
\end{lemma}
\begin{proof}
It is clearly sufficient to prove that, for
$v = v_{\parallel}+v_{\perp}$, $P_Qv_{\parallel}=v_{\parallel}$,
$P_Q^{\perp}v_{\perp}=v_{\perp}$, we have
\begin{eqnarray}
Ev_{\parallel} =  Y(Q^*Q)^{-1}Q^*v_{\parallel}\, ,\;\;\;\;\;
Ev_{\perp} = M(M^*M)^{-1}X^*v_{\perp}\, .
\end{eqnarray}
The first identity is an easy consequence of the fact that
$X^*Q= M^*AQ = M^*Y$, while the second one follows immediately from $Q^*v_{\perp}=0$.
\end{proof}

The following fact (see Appendix  \ref{sec:Svalues_t} for a proof)
will be used several times.
\begin{lemma}\label{lemma:Svalues_t}
For any $t$ there exists $c>0$ such that, for $R\in \{ Q^*Q;\, M^*M;\, X^*X;\, Y^*Y\}$, eventually almost surely as $N\to\infty$,
\begin{eqnarray}
c\le \lambda_{\rm min}(R/N) \le \lambda_{\rm max}(R/N) \le 1/c\, .
\end{eqnarray}
\end{lemma}

Given the above remarks, we will immediately see that Lemma \ref{lemma:MinS}
is implied by the following statement.
\begin{lemma}\label{lemma:ConcreteMinS}
Let $S\subseteq [N]$ be given such that $|S|\le N(\delta-\gamma)$,
for some $\gamma>0$. Then there exists $\alpha_1=\alpha_1(\gamma)>0$
(independent of $t$) and
$\alpha_2=\alpha_2(\gamma,t)>0$ (depending on $t$ and $\gamma$)
such that
\begin{eqnarray*}
\prob\Big\{\min_{\|v\|=1,\,\supp(v)\subseteq S}\big\| Ev +
P_M^{\perp}\tA P_Q^{\perp}v\big\|\le \alpha_2\,\Big|\,
\sigal{S}_t\Big\}\le \, e^{-N\alpha_1}\, ,
\end{eqnarray*}
eventually almost surely as $N\to\infty$.
(With $Ev = Y(Q^*Q)^{-1}Q^*P_Qv + M(M^*M)^{-1}X^*P_Q^{\perp}v$.)
\end{lemma}
In the next section we will show that this lemma implies Lemma
\ref{lemma:MinS}. We will then prove the lemma just stated.
%
%
\subsubsection{Lemma \ref{lemma:ConcreteMinS} implies Lemma \ref{lemma:MinS}}

By Borel-Cantelli, it is sufficient to show
 that, for $S$ measurable on $\sigal{S}_t$ and
$|S|\le N(\delta-c)$ there exist $a_1=a_1(c)>0$ and $a_2=a_2(c,t)>0$, such that
\begin{eqnarray*}
\prob\Big\{\min_{|S'|\le a_1N}\;
\min_{\|v\|=1,\supp(v)\subseteq S\cup S'}\|Av\|< a_2\Big\} \le \frac{1}{N^2}\,,
\end{eqnarray*}
for all $N$ large enough.
Conditioning on $\sigal{S}_t$ and using the union bound, this
probability can be estimated as
\begin{align*}
\E\Big\{\prob\Big\{\min_{|S'|\le a_1N}\;&
\min_{\|v\|=1,\supp(v)\subseteq S\cup S'}\|Av\|< a_2\Big|\,\sigal{S}_t\Big\}
\Big\}\le\\
&\le e^{Nh(a_1)}\E\Big\{\max_{|S'|\le a_1N}\prob\Big\{
\min_{\|v\|=1,\supp(v)\subseteq S\cup S'}\|Av\|< a_2\Big|\,\sigal{S}_t\Big\}
\Big\}\,,
\end{align*}
where $h(p)=-p\log p-(1-p)\log(1-p)$ is the binary entropy function.
The union bound calculation indeed proceeds as follows
\begin{align*}
\prob\{\min_{|S'|\le Na_1}\sX_{S'}<a_2\big|{\sigal{S}_t}\}&\leq \sum_{|S'|\le Na_1}\prob\{\sX_{S'}<a_2\big|{\sigal{S}_t}\}\\
&\leq \Big[\sum_{k=1}^{Na_1}{N\choose k}\Big]\max_{|S'|\le Na_1}\prob\{\sX_{S'}<a_2\big|{\sigal{S}_t}\} \\
&\le e^{Nh(a_1)}\max_{|S'|\le Na_1}\prob\{\sX_{S'}<a_2\big|{\sigal{S}_t}\}\,,
\end{align*}
where $\sX_{S'}=\min_{\|v\|=1,\supp(v)\subseteq S\cup S'}\|Av\|$.  Now, fix $a_1<c/2$ in such a way that $h(a_1)\le \alpha_1(c/2)/2$
(with $\alpha_1$ defined as per Lemma \ref{lemma:ConcreteMinS}).
Further choose $a_2=\alpha_2(c/2,t)/2$.
The above probability is then upper bounded by
\begin{align*}
e^{N\alpha_1(c/2)/2}\;\E\Big\{\max_{|S''|\le N(\delta-c/2)}\prob\Big\{
\min_{\|v\|=1,\supp(v)\subseteq S''}\|Av\|<\frac{1}{2}
\alpha_2(c/2,t)\Big|\,\sigal{S}_t\Big\}
\Big\} \, .
\end{align*}
Finally, applying Lemma \ref{lemma:ConcreteMinS} and
using Lemma \ref{lemma:AvApprox} to estimate $Av$, we get,
for all $N$ large enough,
\begin{align*}
e^{N\alpha_1/2}\;\E\big\{\max_{|S''|\le N(\delta-c/2)} e^{-N\alpha_1}
\big\}\le \frac{1}{N^2} \, .
\end{align*}
This finishes the proof.
\eprooft
%
%
\subsubsection{Proof of Lemma \ref{lemma:ConcreteMinS}}

We begin with the following Pythagorean inequality.
\begin{lemma}\label{lemma:Pitagora}
Let $S\subseteq [N]$ be given such that $|S|\le N(\delta-\gamma)$,
for some $\gamma>0$. Recall that
$Ev =Y(Q^*Q)^{-1}Q^*P_Qv + M(M^*M)^{-1}X^*P_Q^{\perp}v$
and consider the event
\begin{eqnarray*}
\cE_1 \equiv
\Big\{ \big\| Ev +
P_M^{\perp}\tA P_Q^{\perp}v\big\|^2\ge  \frac{\gamma}{4\delta}\big\| Ev
-P_M\tA P_Q^{\perp}v\big\|^2+
\frac{\gamma}{4\delta}\big\|\tA P_Q^{\perp}v\big\|^2\, \;\forall v\;
\mbox{ s.t. }\; \|v\|=1\mbox{ and }\;\supp(v)\subseteq S\Big\}.
\end{eqnarray*}
Then there exists $a=a(\gamma)>0$ such that $\prob\{\cE_1|\sigal{S}_t\}\ge 1-e^{-Na}$.
\end{lemma}
\begin{proof}
We claim that the following inequality holds for all $v\in\reals^N$, that satisfy
$\|v\|=1$ and $\supp(v)\subseteq S$, with the probability claimed in the statement
\begin{eqnarray}
|(Ev-P_M\tA P_Q^{\perp}v \, ,\, \tA P_Q^{\perp}v)| &\le &
\sqrt{1-\frac{\gamma}{2\delta}}\, \|Ev-P_M\tA P_Q^{\perp}v \|
\, \|\tA P_Q^{\perp}v\|\,  .\label{claim:ScalarProd}
\end{eqnarray}
Here the notation $(u,v)$ refers to the usual scalar product $u^*v$ of vectors $u$ and $v$ of the same dimension.
Assuming that the claim holds, we have indeed
\begin{align*}
\big\| Ev +
P_M^{\perp}\tA P_Q^{\perp}v\big\|^2&\ge
\big\| Ev-P_M\tA P_Q^{\perp}v \big\|^2+\big\| \tA P_Q^{\perp}v\big\|^2-2
|(Ev-P_M\tA P_Q^{\perp}v \, ,\, \tA P_Q^{\perp}v)|\\
&\ge\big\| Ev \big\|^2+\big\| P_M^{\perp}\tA P_Q^{\perp}v\big\|^2-2
\sqrt{1-\frac{\gamma}{2\delta}}\, \|Ev -P_M\tA P_Q^{\perp}v\| \,
\|\tA P_Q^{\perp}v\|\\
&\ge \Big(1-\sqrt{1-\frac{\gamma}{2\delta}}\Big)
\Big\{\big\| Ev-P_M\tA P_Q^{\perp}v \big\|^2+\big\| \tA P_Q^{\perp}v\big\|^2\Big\}\, ,
\end{align*}
which implies the thesis.

In order to prove the claim \eqref{claim:ScalarProd},
we notice that for any $v$, the  unit vector
$\tA P_Q^{\perp}v /\|\tA P_Q^{\perp}v \|$ belongs to the random linear space
$\im(\tA P_Q^{\perp}P_S)$. Here $P_S$ is the orthogonal projector
onto the subspace of vectors supported on $S$.
Further $\im(\tA P_Q^{\perp}P_S)$ is a uniformly random subspace of
dimension at most $N(\delta-\gamma)$. Also, the normalized vector
$(Ev-P_M\tA P_Q^{\perp}v)/\|Ev-P_M\tA P_Q^{\perp}v \|$ belongs to the
linear space of dimension at most $2t$ spanned the columns of $M$
and of $B$. The claim follows then from a standard concentration-of-measure
argument.  In particular applying Proposition \ref{propo:Concentration} for
\[
m=n,~~m\lambda=N(\delta-\gamma),~~d=2t~~{ and }~~\ve=\sqrt{1-\frac{\gamma}{2\delta}}-\sqrt{1-\frac{\gamma}{\delta}}
\]
yields
\[
\left(\frac{Ev-P_M\tA P_Q^{\perp}v}{\|Ev-P_M\tA P_Q^{\perp}v \|}\,,~\frac{\tA P_Q^{\perp}v }{\|\tA P_Q^{\perp}v \|}\right)\leq \sqrt{\lambda}+\ve=\sqrt{1-\frac{\gamma}{2\delta}}\,.
\]
(Notice that in Proposition \ref{propo:Concentration} is stated for the equivalent case of a random
sub-space of fixed dimension $d$, and a subspace of dimension scaling linearly with the ambient one.)
\end{proof}
Next we estimate the term $\|\tA P_Q^{\perp}v\|^2$ in the above lower bound.
\begin{lemma}\label{lemma:Event2}
Let $S\subseteq [N]$ be given such that $|S|\le N(\delta-\gamma)$,
for some $\gamma>0$. Then there exists constant $c_1=c_1(\gamma)$,
$c_2 = c_2(\gamma)$ such that the event
\begin{eqnarray*}
\cE_2 \equiv
\Big\{ \big\|\tA P_Q^{\perp}v\big\| \ge c_1(\gamma)
\|P_Q^{\perp}v\big\| \, \;\;\forall v\;
\mbox{ such that }\; \;\supp(v)\subseteq S\Big\}\, ,
\end{eqnarray*}
holds with probability $\prob\{\cE_2|\sigal{S}_t\}\ge 1-e^{-Nc_2}$.
\end{lemma}
\begin{proof}
Let $V$ be the linear space $V=\im(P_Q^{\perp}P_S)$.
Of course the dimension of $V$ is at most $N(\delta-\gamma)$.
Then we have  (for all vectors with  $\supp(v)\subseteq S$)
\begin{eqnarray}
\big\|\tA P_Q^{\perp}v\big\| \ge \sigma_{\rm min}(\tA|_{V}) \, \|P_Q^{\perp}v\big\|\, ,
\end{eqnarray}
where $\tA|_V$ is the restriction of $\tA$ to the subspace $V$.
By invariance of the distribution of $\tA$ under rotation,
$\sigma_{\rm min}(\tA|_{V})$ is distributed as the minimum singular value
of a gaussian matrix of dimensions $N\delta\times {\rm dim}(V)$.
The latter is almost surely bounded away from $0$
as $N\to\infty$,
 since ${\rm dim}(V)\le N(\delta-\gamma)$ (see for instance Theorem
\ref{prop:marchenko-pastur}). Large deviation estimates
\cite{LitvakEtAl} imply that the probability that the minimum singular
value is smaller than a constant $c_1(\gamma)$ is exponentially small.
\end{proof}

Finally a simple bound to control the norm of $Ev$.
\begin{lemma}\label{lemma:EvBound}
There exists a constant $c = c(t)>0$ such that, defining the event,
\begin{eqnarray}
\cE_3 \equiv
\big\{\|EP_Qv\|\ge c(t)\|P_Qv\|\, ,
\|EP_Q^{\perp}v\|\le c(t)^{-1}\|P^{\perp}_Qv\|
,\;\mbox{ for all }\;v\in \reals^N
\big\}\, ,
\end{eqnarray}
we have that $\cE_3$ holds eventually almost surely as $N\to\infty$.
\end{lemma}
\begin{proof}
Without loss of generality take $v = Qa$ for $a\in\reals^t$.
By Lemma \ref{lemma:AvApprox} we have
$\|EP_Qv\|^2 = \|Ya\|^2\ge \lambda_{\rm min}(Y^*Y)\|a\|^2$.
Analogously $\|P_Qv\|^2= \|Qa\|^2\le  \lambda_{\rm max}(Q^*Q)\|a\|^2$.
The bound  $\|EP_Qv\|\ge c(t)\|P_Qv\|$ follows
then from Lemma \ref{lemma:Svalues_t}.

The bound $\|EP_Q^{\perp}v\|\le c(t)^{-1}\|P^{\perp}_Qv\|$
is proved analogously.
\end{proof}

We  can now prove Lemma \ref{lemma:ConcreteMinS} as promised.
\begin{proof}[Proof of Lemma \ref{lemma:ConcreteMinS}]
By Lemma \ref{lemma:EvBound} we can assume that event $\cE_3$ holds,
for some function $c= c(t)$ (without loss of generality $c<1/2$).
We will let $\cE$ be the event
\begin{eqnarray}
\cE\equiv\Big\{ \min_{\|v\|=1,\,\supp(v)\subseteq S}\big\| Ev +
P_M^{\perp}\tA P_Q^{\perp}v\big\|\le \alpha_2(t)\Big\}\, .
\end{eqnarray}
for $\alpha_2(t)>0$ small enough.

Let us assume first that $\|P_Q^{\perp}v\|\le c^2/10$,
whence
\begin{align*}
\| E v-P_M\tA P_Q^{\perp}\| &\ge \|EP_Qv\|-\|EP_Q^{\perp}v\|- \|P_M\tA P_Q^{\perp}v\|\\
&\ge c\|P_Qv\| - (c^{-1}+\|\tA\|_2)\|P_Q^{\perp}v\|\\
&\ge \frac{c}{2}-\frac{c}{10}-\|\tA\|_2\frac{c^2}{10} =
\frac{2c}{5}-\|\tA\|_2\frac{c^2}{10}\,,
\end{align*}
where the last inequality uses $\|P_Qv\|=\sqrt{1-\|P_Q^{\perp}v\|^2}\ge 1/2$.
Therefore, using Lemma \ref{lemma:Pitagora}, we get
\begin{align*}
\prob\{\cE|\sigal{S}_t\}\le
\prob\Big\{\frac{2c}{5}-\|\tA\|_2\frac{c^2}{10}\le \sqrt{\frac{4\delta}{\gamma}}
\alpha_2(t)\Big|
\sigal{S}_t\Big\}+ e^{-Na}\, ,
\end{align*}
and the thesis follows from large deviation bounds on the
norm $\|\tA\|_2$ \cite{Ledoux} by first taking  $c$
small enough,
and then choosing $\alpha_2(t)<\frac{c}{5}\sqrt{\frac{\gamma}{4\delta}}$.

Next we assume $\|P_Q^{\perp}v\|\ge c^2/10$.
Due to Lemma \ref{lemma:Pitagora} and \ref{lemma:Event2}
we can assume that events $\cE_1$ and $\cE_2$ hold. Therefore
\begin{eqnarray*}
\big\| Ev +
P_M^{\perp}\tA P_Q^{\perp}v\big\|\ge \Big(\frac{\gamma}{4\delta}\Big)^{1/2}
\|\tA P_Q^{\perp}v\big\|\ge \Big(\frac{\gamma}{4\delta}\Big)^{1/2}c_1(\gamma)
\|P_Q^{\perp}v\|\,,
\end{eqnarray*}
which proves our thesis.
\end{proof}
%
%
\subsection{Proof of Lemma \ref{lemma:ConvergenceSupport}}
\label{sec:ConvergenceSupport}

The key step consists in establishing the following result,
which will be instrumental in the proof of Lemma \ref{lemma:Convergence}
as well (and whose proof is deferred to Appendix \ref{app:TauPlus}).
\begin{lemma}\label{lemma:Tau_Plus}
Assume $\alpha>\alpha_{\rm min}(\delta)$ and let $\{\covz_{s,t}\}$
be defined by the recursion (\ref{eq:2-times-SE}) with initial condition
(\ref{eq:Initial-2times-SE}). Then there exists constants $\finite_1$,
$\ra_1>0$ such that for all $t\ge 0$
\begin{eqnarray}
\big|\covz_{t,t}-\tau_*^2\big|& \le &\finite_1\, e^{-\ra_1\, t}\, ,
\label{eq:Exponential_1time}\\
\big|\covz_{t,t+1}-\tau_*^2\big|&\le &\finite_1\, e^{-\ra_1\, t}\, .
\label{eq:Exponential_2times}
\end{eqnarray}
\end{lemma}

It is also useful to prove the following fact.
\begin{lemma}\label{lemma:PosDef}
For any $\alpha>0$ and $T\ge 0$, the $T\times T$ matrix
$R_{T+1}\equiv\{\covz_{s,t}\}_{0\le s,t< T}$ is strictly positive definite.
\end{lemma}
\begin{proof}
In proof of Theorem \ref{prop:state-evolution-2times} we show that
\[
\covz_{s,t}=\lim_{N\to\infty}\<h^{s+1},h^{t+1}\>=\lim_{N\to\infty}\<m^{s},m^{t}\>\,,
\]
almost surely.  Hence, $R_{T+1}\asequal\delta\lim_{N\to\infty} (M_{T+1}^*M_{T+1}/N)$. Thus the result follows from Lemma \ref{lemma:Svalues_t}.
%
\end{proof}

It is then relatively easy to deduce the following.
\begin{lemma}\label{lemma:Tau_Plus_ell}
Assume $\alpha>\alpha_{\rm min}(\delta)$ and let $\{\covz_{s,t}\}$
be defined by the recursion (\ref{eq:2-times-SE}) with initial condition
(\ref{eq:Initial-2times-SE}). Then
there exists constants $\finite_2$,
$\ra_2>0$ such that for all $t_1,t_2\ge t\ge 0$
\begin{eqnarray}
\big|\covz_{t_1,t_2}-\tau_*^2\big|\le \finite_2\, e^{-\ra_2\, t}\, .
\end{eqnarray}
\end{lemma}
\begin{proof}
By triangular inequality and Eq.~(\ref{eq:Exponential_1time}), we have
\begin{align}
\big|\covz_{t_1,t_2}-\tau_*^2\big|\le \frac{1}{2}
\big|\covz_{t_1,t_1}-2\covz_{t_1,t_2}+\covz_{t_2,t_2}\big|+
\finite_1\, e^{-\ra_1\, t}\, .\label{eq:ExpTriang}
\end{align}
By  Lemma \ref{lemma:PosDef} there exist gaussian random variables
$Z_0,Z_1,Z_2,\dots$ on the same probability space with $\E\{Z_t\}=0$
and $\E\{Z_tZ_s\}= \covz_{t,s}$ (in fact in proof of Theorem \ref{prop:state-evolution-2times} we show that $\{Z_i\}_{T\ge i\ge0}$ is the weak limit of the empirical distribution of $\{h^{i+1}\}_{T\ge i\ge0}$). Then (assuming, without loss of generality,
$t_2>t_1$) we have
\begin{align*}
\big|\covz_{t_1,t_1}-2\covz_{t_1,t_2}+\covz_{t_2,t_2}\big|
&=\E\{(Z_{t_1}-Z_{t_2})^2\}\\
&= \sum_{i,j=t_1}^{t_2-1}\E\{(Z_{i+1}-Z_{i})(Z_{j+1}-Z_{j})\}\\
&\le \Big[\sum_{i=t_1}^{t_2-1}\E\{(Z_{i+1}-Z_{i})^2\}^{1/2}\Big]^2\\
&\le 4\finite_1\Big[\sum_{i=t_1}^{\infty}e^{-\ra_1i/2}\Big]^{2}\\
&\le \frac{4\finite_1}{(1-e^{-\ra_1/2})^2}\;
e^{-\ra_1t_1}\, ,
\end{align*}
which, together with Eq.~(\ref{eq:ExpTriang}) proves our claim.
\end{proof}

We are now in position to prove Lemma \ref{lemma:ConvergenceSupport}.
\begin{proof}[Proof of Lemma \ref{lemma:ConvergenceSupport}]
We will show that, under the assumptions of the Lemma,
$\lim_{N\to\infty}|S_{t_2}(\gamma)\setminus S_{t_1}(\gamma)|/N\le \xi$
almost surely, which implies our claim.
Indeed, by Theorem \ref{prop:state-evolution-2times} we have
\begin{align*}
\lim_{N\to\infty}\frac{1}{N}&\,|S_{t_2}(\gamma)\setminus S_{t_1}(\gamma)|
= \lim_{N\to\infty}\frac{1}{N}\sum_{i=1}^N\ind_{\big\{|v_i^{t_2}|\ge 1-\gamma,
\;|v_i^{t_1}|< 1-\gamma\big\}}\\
&= \lim_{N\to\infty}\frac{1}{N}\sum_{i=1}^N
\ind_{\big\{|x^{t_2-1}+A^*z^{t_2-1}-x^{t_2}|\ge (1-\gamma)\theta_{t_2-1},
\;|x^{t_1-1}+A^*z^{t_1-1}-x^{t_1}|< (1-\gamma)\theta_{t_2-1}\big\}}\\
& = \prob\big\{|X_0+Z_{t_2-1}|\ge (1-\gamma)\theta_{t_2-1},\;
|X_0+Z_{t_1-1}|< (1-\gamma)\theta_{t_1-1}\big\}\equiv P_{t_1,t_2}\, ,
\end{align*}
where $(Z_{t_1},Z_{t_2})$ are jointly normal with
$\E\{Z_{t_1}^2\} = \covz_{t_1,t_1}$, $\E\{Z_{t_1}Z_{t_2}\} = \covz_{t_1,t_2}$,
$\E\{Z_{t_2}^2\} = \covz_{t_2,t_2}$.
(Notice that, although the function $\ind\{\,\cdots\,\}$ is discontinuous,
the random vector $(X_0+Z_{t_1-1},X_0+Z_{t_2-1})$ admits a density and hence
Theorem \ref{prop:state-evolution-2times} applies by weak convergence
of the empirical distribution of
$\{(x_{i}^{t_1-1}+(A^*z^{t_1-1})_i\,,\,\,x_{i}^{t_2-1}+(A^*z^{t_2-1})_i)\}_{1\le i\le N}$.)

Let $a\equiv (1-\gamma)\alpha\tau_{*}$. By Proposition
\ref{propo:UniqFP}, for any $\ve >0$
and all  $t_*$ large enough
we have $|(1-\gamma)\theta_{t_{i}-1}-a|\le \ve$ for $i\in\{1,2\}$. Then
\begin{align*}
P_{t_1,t_2}&\le  \prob\big\{|X_0+Z_{t_2-1}|\ge a-\ve,\;
|X_0+Z_{t_1-1}|< a+\ve\big\}\\
&\le
\prob\big\{|Z_{t_1-1}-Z_{t_2-1}|\ge 2\ve\big\}+
\prob\{a-3\ve\le|X_0+Z_{t_{1}-1}|\le a+\ve\big\}\\
&\le \frac{1}{4\ve^2}[\covz_{t_1-1,t_1-1}-2\covz_{t_1-1,t_2-1}+\covz_{t_2-1,t_2-1}]+
\frac{4\ve}{\sqrt{2\pi\covz_{t_1-1,t_1-1}}}\\
&\le \frac{1}{\ve^2}\finite_2\, e^{-\ra_2t_*}
+\frac{\ve}{\tau_*}\, ,
\end{align*}
where the last inequality follows by Lemma \ref{lemma:Tau_Plus_ell}.
By taking $\ve = e^{-\ra_2\, t_*/3}$
we finally get (for some constant $C$) $P_{t_1,t_2}\le C\,  e^{-\ra_2t_*}$,
which implies our claim.
\end{proof}
%
%

\section*{Acknowledgement}

It is a pleasure to thank David Donoho and Arian Maleki for many
stimulating exchanges. We are also indebted with Jos\'e Bento who collaborated
in preparing Figures \ref{fig:geneExpression} to \ref{fig:PM1Matrices}.

An earlier version of this paper stated some auxiliary lemmas in terms
of convergence \emph{in probability}. We rectified this to
convergence \emph{almost sure} as for the main theorems (with
virtually no change in the proofs). 
We are grateful to Edgar Dobriban and Weijie Su for pointing out this inconsistency.

This work was partially supported by a Terman fellowship,
the NSF CAREER award CCF-0743978 and the NSF grant DMS-0806211.

\appendix
%
%
\section{Properties of the state evolution recursion}

\subsection{Proof of Proposition \ref{propo:UniqFP}}
\label{app:UniqFP}

It is a straightforward calculus exercise to compute the partial derivatives
\begin{align}
\frac{\partial \seF}{\partial\tau^2}(\tau^2,\theta) &=
\frac{1}{\delta}\, \E\Big\{\Phi\Big(\frac{X_0-\theta}{\tau}\Big)+
\Phi\Big(\frac{-X_0-\theta}{\tau}\Big)\Big\}-
\frac{1}{\delta}\,\E\Big\{\frac{X_0}{\tau}\,\phi\Big
(\frac{X_0-\theta}{\tau}\Big)-\frac{X_0}{\tau}\,
\phi\Big(\frac{-X_0-\theta}{\tau}\Big)\Big\}\, ,
\label{eq:deFtau2}\\
\frac{\partial \seF}{\partial\theta}(\tau^2,\theta) &=
\frac{2\theta}{\delta}\, \E\Big\{\Phi\Big(\frac{X_0-\theta}{\tau}\Big)+
\Phi\Big(\frac{-X_0-\theta}{\tau}\Big)\Big\}-
\frac{2\tau}{\delta}\,\E\Big\{\phi\Big
(\frac{X_0-\theta}{\tau}\Big)+
\phi\Big(\frac{-X_0-\theta}{\tau}\Big)\Big\}\, .
\label{eq:deFtheta}
\end{align}
From these formulae we obtain the total derivative
\begin{eqnarray}
\delta\, \frac{\de \seF}{\de\tau^2}(\tau^2,\alpha\tau) &= &
(1+\alpha^2)\,\E\Big\{\Phi\Big(\frac{X_0-\alpha\tau}{\tau}\Big)+
\Phi\Big(\frac{-X_0-\alpha\tau}{\tau}\Big)\Big\}\label{eq:deF}\\
&&-\E\Big\{\Big(\frac{X_0+\alpha\tau}{\tau}\Big)\,\phi\Big
(\frac{X_0-\alpha\tau}{\tau}\Big)-\Big(\frac{X_0-\alpha\tau}{\tau}\Big)\,
\phi\Big(\frac{-X_0-\alpha\tau}{\tau}\Big)\Big\}\, .\nonumber
\end{eqnarray}
Differentiating once more
\begin{eqnarray}
\delta\frac{\de^2 \seF}{\de(\tau^2)^2}(\tau^2,\alpha\tau) = -\frac{1}{2\tau^2}
\E\Big\{\Big(\frac{X_0}{\tau}\Big)^3\,\Big[\phi\Big
(\frac{X_0-\alpha\tau}{\tau}\Big)-\,
\phi\Big(\frac{-X_0-\alpha\tau}{\tau}\Big)\Big]\Big\}\, .\nonumber
\end{eqnarray}
Now we have
\begin{eqnarray}
u^3[\phi(u-\alpha)-\phi(-u-\alpha)]\ge 0\, ,
\end{eqnarray}
with the inequality being strict whenever $\alpha>0$, $u\neq 0$. It follows
that $\tau^2\mapsto\seF(\tau^2,\alpha\tau)$ is
concave, and strictly concave provided $\alpha>0$ and $X_0$ is not identically
$0$.

From Eq.~(\ref{eq:deF}) we obtain
\begin{eqnarray}
\lim_{\tau^2\to\infty}
\frac{\de \seF}{\de\tau^2}(\tau^2,\alpha\tau) =  \frac{2}{\delta}
\big\{(1+\alpha^2)\Phi(-\alpha)-\alpha\, \phi(\alpha)\big\}\, ,
\end{eqnarray}
which is strictly positive for all $\alpha\ge 0$.
To see this, let
$f(\alpha)\equiv (1+\alpha^2)\Phi(-\alpha)-\alpha\, \phi(\alpha)$,
and notice that $f'(\alpha) = 2\alpha\Phi(-\alpha)-2\phi(\alpha)<0$,
and $f(\infty) = 0$.

Since $\tau^2\mapsto\seF(\tau^2,\alpha\tau)$ is concave, and strictly
increasing for $\tau^2$ large enough, it also follows that it is increasing
everywhere.

Notice that $\alpha\mapsto f(\alpha)$ is strictly decreasing
with $f(0) = 1/2$.
Hence, for  $\alpha>\alpha_{\rm min}(\delta)$, we have
$\seF(\tau^2,\alpha\tau)>\tau^2$ for $\tau^2$ small enough and
$\seF(\tau^2,\alpha\tau)<\tau^2$ for $\tau^2$ large enough.
Therefore the fixed point equation admits at least one solution.
It follows from the concavity of $\tau^2\mapsto\seF(\tau^2,\alpha\tau)$
that the solution is unique and that the sequence of iterates  $\tau_t^2$
converge to $\tau_*$.
\eprooft
%
%
\subsection{Proof of Proposition \ref{propo:Lambda}}
\label{app:Lambda}

As a first step, we claim that $\alpha\mapsto\tau_*^2(\alpha)$
is continuously differentiable on $(0,\infty)$. Indeed this is defined as
the unique solution of
\begin{eqnarray}
\tau_*^2 = \seF(\tau^2_*,\alpha\tau_*)\, .
\end{eqnarray}
Since $(\tau^2,\alpha)\mapsto \seF(\tau^2_*,\alpha\tau_*)$ is continuously
differentiable and $ 0\le\frac{\de\seF}{\de\tau^2}(\tau^2_*,\alpha\tau_*)<1$
(the second inequality being a consequence of concavity plus
$\lim_{\tau^2\to\infty}\frac{\de\seF}{\de\tau^2}(\tau^2,\alpha\tau)<1$,
both shown in the proof of Proposition \ref{propo:UniqFP}),
the claim follows from the implicit function theorem applied to the mapping
$(\tau^2,\alpha)\mapsto [\tau^2-F(\tau^2,\alpha)]$.

Next notice that $\tau_*^2(\alpha)\to +\infty$ as $\alpha\downarrow \alpha_{\rm min}(\delta)$. Indeed, introducing the notation $\seF'_{\infty}\equiv
 \lim_{\tau^2\to\infty}\frac{\de\seF}{\de\tau^2}(\tau^2,\alpha\tau)$,
we have, again by concavity,
\begin{align*}
\tau_*^2 \ge \seF(0,0)+ \seF'_{\infty}\tau_*^2\, ,
\end{align*}
i.e. $\tau_*^2\ge \seF(0,0)/(1- \seF'_{\infty})$. Now
$ \seF(0,0)\ge\sigma^2$, while $\seF'_{\infty}\uparrow 1$
as $\alpha\downarrow\alpha_{\rm min}(\delta)$
(shown in the proof of Proposition \ref{propo:UniqFP}),
whence the claim follows.

Finally $\tau_*^2(\alpha)\to \sigma^2+\E\{X_0^2\}/\delta$ as $\alpha\to\infty$.
Indeed for any fixed $\tau^2>0$ we have $\seF(\tau^2,\alpha\tau)\to
\sigma^2+\E\{X_0^2\}/\delta$ as $\alpha\to\infty$ whence the claim follows
by uniqueness of $\tau_*$.

Next consider the function $(\alpha,\tau^2)\mapsto g(\alpha,\tau^2)$
defined by
\begin{eqnarray*}
g(\alpha,\tau^2)\equiv \alpha\tau\Big\{1 -
\frac{1}{\delta}\, \prob\{|X_0+\tau\,Z|\ge\alpha\tau\}
\Big\} \, .
\end{eqnarray*}
Notice that $\lambda(\alpha) = g(\alpha,\tauinf^2(\alpha))$.
Since $g$ is continuously differentiable, it follows that
$\alpha\mapsto \lambda(\alpha)$ is continuously differentiable as well.

Next consider $\alpha\downarrow\alpha_{\rm min}$, and let
$l(\alpha)\equiv 1 - \frac{1}{\delta}\,
\prob\{|X_0+\tau_*\,Z|\ge\alpha\tau_*\}$. Since $\tau_*\to +\infty$
in this limit, we have
\begin{eqnarray*}
l_*\equiv\lim_{\alpha\to\alpha_{\rm min}+}l(\alpha)= 1 - \frac{1}{\delta}\,
\prob\{|Z|\ge\alpha_{\rm min}\} = 1 - \frac{2}{\delta}\, \Phi(-\alpha_{\rm min})
\, .
\end{eqnarray*}
Using the characterization of $\alpha_{\rm min}$ in Eq.~\eqref{eq:AlphaMin}
(and the well known inequality $\alpha\Phi(-\alpha)\le\phi(\alpha)$
valid for all $\alpha>0$), it is immediate to show that
$l_*<0$. Therefore
\begin{eqnarray*}
\lim_{\alpha\to\alpha_{\rm min}+}\lambda(\alpha) = l_*\lim_{\alpha\to\alpha_{\rm min}+}
\alpha\tau_*(\alpha) = -\infty\, .
\end{eqnarray*}

Finally let us consider the limit $\alpha\to\infty$.
Since $\tau_*(\alpha)$ remains bounded, we have
$\lim_{\alpha\to\infty}\prob\{|X_0+\tau_*\,Z|\ge\alpha\tau_*\}=0$
whence
\begin{eqnarray*}
\lim_{\alpha\to\infty}\lambda(\alpha) = \lim_{\alpha\to\infty}
\alpha\tau_*(\alpha) = \infty\, .
\end{eqnarray*}
\eprooft
%
%
\subsection{Proof of Corollary \ref{coro:AlphaUnique}}
\label{app:LambdaBis}

By Proposition \ref{propo:Lambda}, it is sufficient to prove that,
for any $\lambda>0$ there exists a unique $\alpha>\alpha_{\rm min}$
such that $\lambda(\alpha) = \lambda$.
Assume by contradiction that there are two distinct
such values $\alpha_1$, $\alpha_2$.

Notice that in this case, the function $\alpha(\lambda)$
is not defined uniquely and we can apply Theorem \ref{thm:Risk}
to both choices $\alpha(\lambda)=\alpha_1$ and $\alpha(\lambda) = \alpha_2$.
Using the test function $\psi(x,y) = (x-y)^2$ we deduce that
\begin{eqnarray*}
\lim_{N\to\infty}\frac{1}{N}\|\hx-x_0\|^2 = \E\big\{[\eta(X_0+\tau_*Z\,;\,
\alpha\tau_*)-X_0]^2\big\} = \delta(\tau_*^2-\sigma^2)\, .
\end{eqnarray*}
Since the left hand side does not depend on the choice of $\alpha$,
it follows that $\tau_*(\alpha_1) = \tau_*(\alpha_2)$.

Next apply Theorem \ref{thm:Risk} to
the function  $\psi(x,y) = |x|$. We get
\begin{eqnarray*}
\lim_{N\to\infty}\frac{1}{N}\|\hx\|_1 = \E\big\{|\eta(X_0+\tau_*Z\,;\,
\alpha\tau_*)|\big\} \, .
\end{eqnarray*}
For fixed $\tau_*$, $\theta\mapsto\E\big\{|\eta(X_0+\tau_*Z\,;\,
\theta)|\big\}$ is strictly decreasing in $\theta$. It follows that
$\alpha_1\tau_*(\alpha_1)=\alpha_2\tau_*(\alpha_2)$.
Since we already proved that  $\tau_*(\alpha_1) = \tau_*(\alpha_2)$,
we conclude $\alpha_1=\alpha_2$.
\eprooft
%
%
\section{Proof of Theorem \ref{prop:state-evolution-2times}}
\label{app:state-evolution-2times}

First note that using representation \eqref{eq:h-as-z-and-x} we have $x^t+A^*z^t=x_0-h^{t+1}$.
Furthermore, using Lemma \ref{lem:elephant}(b)  we have almost surely
\begin{align*}
\lim_{N\to\infty}\frac{1}{N}\sum_{i=1}^N\psi
\big(x_{0,i}-h_i^{s+1},x_{0,i}-h_{i}^{t+1},x_{0,i}\big) &=
\E\Big\{\psi\big(X_0-\tZ_s,X_0-\tZ_t,X_0\big)\Big\}\,\\
&=\E\Big\{\psi\big(X_0+\tZ_s,X_0+\tZ_t,X_0\big)\Big\}\,\\
\end{align*}
for gaussian variables $\tZ_s$, $\tZ_t$ that have zero mean and are independent of $X_0$.
Define for all $s\ge0$ and $t\ge0$,
\begin{align}\label{eq:R=Rhat}
\tcovz_{t,s}\equiv \lim_{N\to\infty}\<h^{t+1},h^{s+1}\>=\E\{\tZ_t\tZ_s\}\,.
\end{align}
Therefore, all we need to show is that for all $s,t\ge0$: $\covz_{t,s}$ and $\tcovz_{t,s}$
are equal. We prove this by induction on $\max(s,t)$.

\begin{itemize}
\item For $s=t=0$ we have using Lemma \ref{lem:elephant}(b)  almost surely
\begin{align*}
\tcovz_{0,0}\equiv \lim_{N\to\infty}\<h^{1},h^{1}\>=\tau_0^2=\sigma^2+\frac{1}{\delta}\E\{X_0^2\}\,,
\end{align*}
that is equal to $\covz_{0,0}$.

\item \emph{Induction hypothesis:} Assume that for all $s\le k$ and $t\le k$,
\begin{align}\label{eq:Rst-induc-hyp}
\covz_{t,s}=\tcovz_{t,s}\,.
\end{align}

\item Then we prove Eq. \eqref{eq:Rst-induc-hyp} for $t=k+1$ (case $s=k+1$ is similar). First assume $s=0$ and $t=k+1$ in which using Lemma \ref{lem:elephant}(c) we have almost surely
\begin{align*}
\tcovz_{k+1,0}&=\lim_{N\to\infty}\<h^{k+2},h^{1}\>=\lim_{n\to\infty}\<m^{k+1},m^{0}\>\\
&=\lim_{n\to\infty}\<b^{k+1}-w,b^{0}-w\>=\sigma^2+\frac{1}{\delta}\lim_{N\to\infty}\<q^{k+1},q^{0}\>\\
&=\sigma^2+\frac{1}{\delta}\E\left\{[\eta(X_0-\tZ_{k};\theta_{k})-X_0][-X_0]\right\}\,,\\
&=\sigma^2+\frac{1}{\delta}\E\left\{[\eta(X_0+\tZ_{k};\theta_{k})-X_0][-X_0]\right\}\,,
\end{align*}
where the last equality uses $q^0=-x_0$ and Lemma \ref{lem:elephant}(b) for the pseudo-Lipschitz function $(h_i^{k+1},x_{0,i})\mapsto [\eta(x_{0,i}-h_i^{k+1};\theta_k)-x_{0,i}][-x_{0,i}]$. Here $X_0\sim p_{X_0}$ and $\tZ_{k}$ are independent and the latter is mean zero gaussian with $\E\{\tZ_{k}^2\}=\tcovz_{k,k}$.
But using the induction hypothesis, $\tcovz_{k,k}=\covz_{k,k}$ holds. Hence,
we can apply Eq. \eqref{eq:Initial-2times-SE} to obtain $\tcovz_{t,0}=\covz_{t,0}$.

Similarly, for the case $t=k+1$ and $s>0$, using Lemma \ref{lem:elephant}(b)(c) we have almost surely
\begin{align*}
\tcovz_{k+1,s}&=\lim_{N\to\infty}\<h^{k+2},h^{s+1}\>=\lim_{n\to\infty}\<m^{k+1},m^{s}\>\\
&=\lim_{n\to\infty}\<b^{k+1}-w,b^{s}-w\>=\sigma^2+\frac{1}{\delta}\lim_{N\to\infty}\<q^{k+1},q^{s}\>\\
&=\sigma^2+\frac{1}{\delta}\E\{[\eta(X_0+\tZ_{k};\theta_{k})-X_0][\eta(X_0+\tZ_{s-1};\theta_{s-1})-X_0]\}\,,
\end{align*}
for $X_0\sim p_{X_0}$ independent of zero mean gaussian variables $\tZ_{k}$ and $\tZ_{s-1}$ that satisfy
\[
\covz_{k,s-1}=\E\{\tZ_{k}\tZ_{s-1}\}\,,~~~\covz_{k,k}=\E\{\tZ_{k}^2\}\,,~~~~\covz_{s-1,s-1}=\E\{\tZ_{s-1}^2\}\,,
\]
using the induction hypothesis.  Hence the result follows.
\end{itemize}
\endproof
%
%
\section{Proof of Lemma \ref{lemma:Convergence}}
\label{app:Convergence}

The proof of Lemma \ref{lemma:Convergence} relies on Lemma
\ref{lemma:Tau_Plus} which we will prove in the first subsection.
%
%
\subsection{Proof of Lemma \ref{lemma:Tau_Plus}}
\label{app:TauPlus}

Before proving Lemma \ref{lemma:Tau_Plus}, we state and prove
the following property of gaussian random variables.
\begin{lemma}
\label{lemma:box-monotonicity}
Let $Z_1$ and $Z_2$ be jointly gaussian random variables with $\E(Z_1^2)=
\E(Z_2^2)=1$ and $\E(Z_1Z_2)=c\ge 0$. Let $I$ be a measurable subset
of the real line.
Then $\prob(Z_1\in I,~Z_2\in I)$ is an increasing function of
$c\in [0,1]$.
\end{lemma}
\begin{proof}
Let $\{X_s\}_{s\in\reals}$ be the standard Ornstein-Uhlenbeck process.
Then $(Z_1,Z_2)$ is distributed as $(X_0,X_t)$ for
$t$ satisfying $c=e^{-2t}$. Hence
\begin{eqnarray}
\prob(Z_1\in I,~Z_2\in I) = \E[f(X_0)f(X_t)] \, ,
\end{eqnarray}
for $f$ the indicator function of $I$. Since the Ornstein-Uhlenbeck process
is reversible with respect to the standard gaussian measure $\mu_{\rm G}$,
we have
\begin{eqnarray}
 \E[f(X_0)f(X_t)] = \sum_{\ell =0}^{\infty} e^{-\lambda_{\ell}t}
\, (\psi_\ell,f)^2_{\mu_{\rm G}} = \sum_{\ell =0}^{\infty} c^{\frac{\lambda_{\ell}}{2}}
\, (\psi_\ell,f)^2_{\mu_{\rm G}}
\end{eqnarray}
with $0\le \lambda_0\le \lambda_1\le \dots$ the eigenvalues of its generator,
$\{\psi_{\ell}\}_{\ell\ge 0}$ the corresponding eigenvectors and
$(\,\cdot\,,\,\cdot\,)_{\mu_{\rm G}}$ the scalar product in
$L^2(\mu_{\rm G})$. The thesis follows.
\end{proof}
We now pass to the proof of  Lemma \ref{lemma:Tau_Plus}.
\begin{proof}[Proof of Lemma \ref{lemma:Tau_Plus}]
It is convenient to change coordinates and define
\begin{eqnarray}
y_{t,1} \equiv \covz_{t-1,t-1}=\tau_{t-1}^2\, ,\;\;\;
y_{t,2} \equiv \covz_{t,t}=\tau_{t}^2\, ,
\;\;\;\;
y_{t,3} \equiv \covz_{t-1,t-1}-2\covz_{t,t-1}+\covz_{t,t}\, .
\end{eqnarray}
The vector $y_t=(y_{t,1},y_{t,2},y_{t,3})$
belongs to $\reals_+^3$ by Lemma \ref{lemma:PosDef}.
Using Eq.~\eqref{eq:2-times-SE}, it is immediate to see that this
is updated according to the mapping
\begin{eqnarray}
y_{t+1} & = & \G(y_t)\, ,\nonumber\\
\G_1(y_t) & \equiv & y_{t,2}\, ,\\
\G_2(y_t) & \equiv & \sigma^2+\frac{1}{\delta}
\E\{[\eta(X_0+Z_t;\alpha \sqrt{y_{t,2}})-X_0]^2\}\, ,
\label{eq:G2}\\
\G_3(y_t) & \equiv & \frac{1}{\delta}
\E\{[\eta(X_0+Z_t;\alpha \sqrt{y_{t,2}})-\eta(X_0+Z_{t-1};\alpha
\sqrt{y_{t,1}})]^2\}\, .\label{eq:G3}
\end{eqnarray}
where $(Z_t,Z_{t-1})$ are jointly gaussian with zero mean and covariance
determined by
$\E\{Z_t^2\}=y_{t,2}$, $\E\{ Z_{t-1}^2\}=y_{t,1}$,
$\E\{ (Z_t-Z_{t-1})^2\}=y_{t,3}$.
This mapping is defined for $y_{t,3}\le 2(y_{t,1}+y_{t,2})$.

Next we will show that by induction on $t$ that the stronger
inequality $y_{t,3}< (y_{t,1}+y_{t,2})$ holds for all $t$.
We have indeed
\begin{eqnarray*}
y_{t+1,1}+y_{t+1,2}-y_{t+1,3} = 2\sigma^2+\frac{2}{\delta}\,
\E\{\eta(X_0+Z_t;\alpha \sqrt{y_{t,2}})\,\eta(X_0+Z_{t-1};\alpha
\sqrt{y_{t,1}})\}\, .
\end{eqnarray*}
Since $\E\{Z_tZ_{t-1}\} = (y_{t,1}+y_{t,2}-y_{t,3})/2$
and $x\mapsto\eta(x;\theta)$ is monotone, we deduce that
 $y_{t,3}< (y_{t,1}+y_{t,2})$ implies
that $Z_t$, $Z_{t-1}$ are positively correlated. Therefore
$\E\{\eta(X_0+Z_t;\alpha \sqrt{y_{t,2}})\,\eta(X_0+Z_{t-1};\alpha
\sqrt{y_{t,1}})\}\ge 0$, which in turn yields
$y_{t+1,3}< (y_{t+1,1}+y_{t+1,2})$.

The initial
condition implied by Eq.~\eqref{eq:Initial-2times-SE} is
\begin{align*}
y_{1,1} & = \sigma^2+\frac{1}{\delta}\,\E\{X_0^2\}\,,\\
y_{1,2} & = \sigma^2+\frac{1}{\delta}\,\E\{[\eta(X_0+Z_0;\theta_0)-X_0]^2\}\,,\\
y_{1,3} & = \frac{1}{\delta}\,\E\{\eta(X_0+Z_0;\theta_0)^2\}\, ,
\end{align*}
It is easy to check that these satisfy
$y_{1,3}<y_{1,1}+y_{1,2}$. (This follows from
$\E\{X_0[X_0-\eta(X_0+Z_0;\theta_0)]\}>0$ because $x_0\mapsto x_0-
\E_Z\eta(x_0+Z_0;\theta_0)$ is monotone increasing.)
We can hereafter therefore assume $y_{t,3}< y_{t,1}+y_{t,2}$
for all $t$.

We will consider the above iteration for arbitrary
initialization $y_0$ (satisfying  $y_{0,3}< y_{0,1}+y_{0,2}$)
and will show the following three facts:
\begin{itemize}
\item[]{\bf Fact $(i)$.} As $t\to\infty$, $y_{t,1},y_{t,2}\to\tau_*^2$.
Further the convergence is monotone.
\item[]{\bf Fact $(ii)$.} If $y_{0,1}=y_{0,2}=\tau_*^2$ and $y_{0,3}\le 2\tau_*^2$,
then $y_{t,1}=y_{t,2}=\tau_*^2$ for all $t$ and $y_{t,3}\to 0$.
\item[]{\bf Fact $(iii)$.} The jacobian $J=J_{\G}(y_*)$ of  $\G$
at $y_* = (\tau_*^2,\tau_*^2,0)$ has spectral radius $\sigma(J)<1$.
\end{itemize}
By simple compactness arguments, Facts $(i)$ and $(ii)$ imply $y_t\to y_*$
as $t\to\infty$. (Notice that $y_{t,3}$ remains bounded
since $y_{t,3}\le (y_{t,1}+y_{t,2})$ and by the convergence
of $y_{t,1},y_{t,2}$.)
Fact $(iii)$ implies that convergence is exponentially
fast.

\vspace{0.2cm}

\emph{Proof of Fact $(i)$.} Notice that $y_{t,2}$ evolves independently
by $y_{t+1,2} = \G_2(y_t) = \seF(y_{2,t},\alpha\sqrt{y_{2,t}})$,
with $\seF(\,\cdot\,,\,\cdot\,)$ the state evolution mapping introduced in
Eq.~\eqref{eq:1-dim-SE}. It follows from Proposition \ref{propo:UniqFP}
that $y_{t,2}\to \tau_*^2$ monotonically for any initial condition.
Since $y_{t+1,1} = y_{t,2}$, the same happens for $y_{t,1}$.

\vspace{0.2cm}

\emph{Proof of Fact $(ii)$.} Consider the function $\G_*(x) =
\G_3(\tau_*^2,\tau_*^2,x)$. This is defined for $x\in [0,4\tau_*^2]$
but since $y_{t,3}<y_{t,1}+y_{t,2}$ we will only consider
$\G_*:[0,2\tau^2_*]\to \reals_+$. Obviously $\G_*(0) = 0$.
Further $\G_*$ can be represented as follows in terms of the
independent random variables $Z$, $W\sim\normal(0,1)$:
\begin{eqnarray}
\G_*(x) = \frac{1}{\delta}
\E\{[\eta(X_0+\sqrt{\tau_*^2-x/4}Z+(\sqrt{x}/2)W;\alpha \tau_*)-\eta(X_0+\sqrt{\tau_*^2-x/4}Z-(\sqrt{x}/2)W;\alpha
\tau_*)]^2\}\, .
\end{eqnarray}
A straightforward calculation yields
\begin{align*}
\G'_*(x) = \frac{1}{\delta}\E\{\eta'(X_0+Z_t; \alpha\tau_*)
\eta'(X_0+Z_{t-1}; \alpha\tau_*)\} =
\frac{1}{\delta}\prob\{|X_0+Z_t|\ge \alpha\tau_*,\;
|X_0+Z_{t-1}|\ge \alpha\tau_*\}\, ,
\end{align*}
where $Z_{t-1}=\sqrt{\tau_*^2-x^2/4}Z+(x/2)W$,
$Z_{t}=\sqrt{\tau_*^2-x^2/4}Z-(x/2)W$.
In particular, by Lemma \ref{lemma:box-monotonicity},
$x\mapsto \G_*(x)$ is strictly increasing (notice that the covariance
of $Z_{t-1}$ and $Z_{t}$ is $\tau_*^2-(x/2)$ which is decreasing in $x$).
Further
\begin{align*}
\G'_*(0) =\frac{1}{\delta}\E\{\eta'(X_0+\tau_*\, Z; \alpha\tau_*)\} \, .
\end{align*}
Hence, since $\lambda>0$ using Eq.~(\ref{eq:calibration}) we have $\G'(0)<1$.
Finally, by Lemma \ref{lemma:box-monotonicity}, $x\mapsto \G'(x)$
is decreasing in $[0,2\tau_*)$. It follows that $y_{t,3}\le \G'(0)^ty_{0,3}\to
0$ as claimed.

\vspace{0.2cm}

\emph{Proof of Fact $(iii)$.}
From the definition of $\G$, we have the following expression for
the Jacobian
\begin{align*}
J_{\G}(y_*) = \left(\begin{array}{ccc}
0 & 1 & 0\\
0 & \seF'(\tau_*^2) & 0\\
a & \G_*'(0) & b\\
\end{array}\right)
\end{align*}
where with an abuse of notation we let $\seF'(\tau_*^2) \equiv
\left.\frac{\de\phantom{\tau^2}}{\de\tau^2}\seF(\tau^2,\alpha\tau)
\right|_{\tau^2=\tau^2_*}$. Computing the eigenvalues of the above matrix,
we get
\begin{align*}
\sigma(J) = \max\big\{\, \seF'(\tau_*^2)\,,\, \G_*'(0) \,\big\}\, .
\end{align*}
Since $\G_*'(0)<1$ as proved above, and $\seF(\tau_*^2)<1$
as per Proposition \ref{propo:UniqFP}, the claim follows.
\end{proof}
%
%
\subsection{Lemma \ref{lemma:Tau_Plus} implies Lemma \ref{lemma:Convergence}}

Using representations (\ref{eq:b-as-z}) and (\ref{eq:q-as-x}) (i.e., $b^t=w-z^t$ and $q^t=x_0-x^t$) and Lemma \ref{lem:elephant}(c) we obtain,
\begin{align*}
\lim_{n\to\infty}\frac{1}{n}\|z^{t+1}-z^t\|_2^2&=\lim_{n\to\infty}\frac{1}{n}\|b^{t+1}-b^t\|_2^2\\
&\asequal\frac{1}{\delta}\lim_{N\to\infty}\frac{1}{N}\|q^{t+1}-q^t\|_2^2\\
&=\frac{1}{\delta}\lim_{N\to\infty}\frac{1}{N}\|x^{t+1}-x^t\|_2^2\,,
\end{align*}
where the last equality uses $q^t=x^t-x_0$.
%
%
%
Therefore, it is sufficient to prove the thesis for $\|x^{t+1}-x^t\|_2$.
By state evolution, Theorem \ref{prop:state-evolution-2times}, we have
\begin{align*}
\lim_{N\to\infty}\frac{1}{N}\|x^{t+1}-x^t\|_2^2& =
\E\big\{\big[\eta(X_0+Z_t;\theta_t)-\eta(X_0+Z_{t-1};\theta_{t-1})\big]^2\big\}\\
&\le 2(\theta_t-\theta_{t-1})^2+2\,\E\{(Z_t-Z_{t-1})^2\}=
2(\theta_t-\theta_{t-1})^2+2(\covz_{t,t}-2\covz_{t,t-1}+\covz_{t-1,t-1})\, .
\end{align*}
The first term vanishes as $t\to\infty$ because
$\theta_t=\alpha\tau_t\to\alpha\tau_*$  by Proposition \ref{propo:UniqFP}.
The second term instead vanishes since
$\covz_{t,t}\to \tau_*$, $\covz_{t,t-1}\to \tau_*$ by Lemma
\ref{lemma:Tau_Plus}.
\endproof
%
%
\section{Proof of Lemma \ref{lemma:Svalues_t}}
\label{sec:Svalues_t}

First note that the upper bound on $\lambda_{\max}(R/N)$ is trivial since using representations
\eqref{eq:X=A*M}, \eqref{eq:Y=AQ}, $q^t=f_t(h^t,x_0)$, $m^t=g_t(b^t,w)$ and Lemma \ref{lem:elephant}(c)(d)
all entries of the matrix $R/N$ are bounded as $N\to\infty$ and the matrix
has fixed dimensions. Hence, we only focus on the lower-bound for $\lambda_{\min}(R/N)$.

The result for $R=M^*M$ and $R=Q^*Q$ follows directly from Lemma \ref{lem:elephant}(g) and Lemma 8 of \cite{BM-MPCS-2010}.

For $R=Y^*Y$ and $R=X^*X$ the proof is by induction on $t$.
\begin{itemize}
\item For $t=1$ we have $Y_t=b^0$ and $X_t=h^1+\xi_0q^0=h^1-x_0$.  Using Lemma
\ref{lem:elephant}(b)(c) we obtain almost surely
\begin{align*}
\lim_{N\to\infty}\frac{Y_t^*Y_t}{N}&=\delta\lim_{n\to\infty}\<b^0,b^0\>=\lim_{N\to\infty}\<q^0,q^0\>=\E\{X_0^2\}\,,\\
\lim_{N\to\infty}\frac{X_t^*X_t}{N}&=\lim_{N\to\infty}\<h^1-x^0,h^1-x^0\>=\E\{(\tau_0Z_0+X_0)^2\}=\sigma^2+\frac{\delta+1}{\delta}\E\{X_0^2\}\,,
\end{align*}
where both are positive by the assumption $\prob\{X_0\neq 0\}>0$.

\item \emph{Induction hypothesis:} Assume that for all $t\leq k$ there exist positive constants $c_X(t)$ and $c_Y(t)$ such that as
$N\to\infty$
\begin{align}
c_Y(t)&\leq\lambda_{\min}(\frac{Y_t^*Y_t}{N}) \,,\label{eq:induc-hyp-lambda_min(Y*Y)}\\
c_X(t)&\leq \lambda_{\min}(\frac{X_t^*X_t}{N}) \,.\label{eq:induc-hyp-lambda_min(X*X)}
\end{align}

\item Now we prove Eq. \eqref{eq:induc-hyp-lambda_min(Y*Y)} for $t=k+1$ (proof of \eqref{eq:induc-hyp-lambda_min(X*X)} is similar).
We will prove that there is a positive constant $c$
such that as $N\to\infty$, for any vector $\va_t\in \reals^t$:
\[
\<Y_t\,\va_t,Y_t\,\va_t\>\ge c\|\va_t\|_2^2\,.
\]
First write $\va_t=(a_1,\ldots,a_t)$ and denote its first $t-1$ coordinates with $\va_{t-1}$. Next, we consider the conditional distribution
$A|_{\sigal{S}_{t-1}}$. Using Eqs. \eqref{eq:Conditional_A} and \eqref{eq:Et} we obtain (since $Y_t=AQ_t$)
\begin{align*}
Y_t\,\va_t|_{\sigal{S}_{t-1}}&\ed A|_{\sigal{S}_{t-1}}(Q_{t-1}\,\va_{t-1}+a_tq^{t-1})\\
& = E_{t-1}(Q_{t-1}\,\va_{t-1}+a_tq^{t-1})+a_tP_{M_{t-1}}^{\perp}\tA q^{t-1}_{\perp}\,.
\end{align*}
Hence, conditional on $\sigal{S}_{t-1}$ we have, almost surely
\begin{align}\label{eq:Rep-Ya}
\lim_{N\to\infty}\<Y_t\,\va_t,Y_t\,\va_t\>
&= \lim_{N\to\infty}\frac{1}{N}\|Y_{t-1}\,\va_{t-1}+a_tE_{t-1}q^{t-1}\|^2+a_t^2 \lim_{N\to\infty}\<q^{t-1}_{\perp},q^{t-1}_{\perp}\>\,.
\end{align}
Here we used the fact that $\tA$ is a random matrix with
i.i.d. $\normal(0,1/n)$ entries independent of $\sigal{S}_{t-1}$
(cf. Lemma \ref{lem:prop-Gaussian-matrix}) which implies that almost surely

- $\lim_{N\to\infty}\<P_{M_{t-1}}^{\perp}\tA q^{t-1}_{\perp},P_{M_{t-1}}^{\perp}\tA q^{t-1}_{\perp}\>=\lim_{N\to\infty}\<q^{t-1}_{\perp},q^{t-1}_{\perp}\>$,

- $\lim_{N\to\infty}\<P_{M_{t-1}}^{\perp}\tA q^{t-1}_{\perp},Y_{t-1}\,\va_{t-1}+a_tb^{t-1}+a_t\lambda_{t-1}m^{t-2}\>=0$.

From Lemma \ref{lem:elephant}(g) we know that $\lim_{N\to\infty}\<q^{t-1}_{\perp},q^{t-1}_{\perp}\>$ is larger than a positive constant $\varsigma_t$. Hence, from representation \eqref{eq:Rep-Ya} and induction hypothesis \eqref{eq:induc-hyp-lambda_min(Y*Y)}
\begin{align*}
\lim_{N\to\infty}\<Y_t\,\va_t,Y_t\,\va_t\>
&\ge \lim_{N\to\infty}\left[\sqrt{c_Y(t-1)}\|\va_{t-1}\|-\frac{|a_t|}{\sqrt{N}}\|b^{t-1}+\lambda_{t-1}m^{t-2}\|\right]^2+a_t^2 \varsigma_t\,.
\end{align*}
To simplify the notation let $c'_t\equiv\lim_{N\to\infty}N^{-1/2}\|b^{t-1}+\lambda_{t-1}m^{t-2}\|$.  Now if $c'_t|a_t|\leq \sqrt{c_Y(t-1)}\|\va_{t-1}\|/2$ then
\begin{align}
\lim_{N\to\infty}\<Y_t\,\va_t,Y_t\,\va_t\>
&\geq \frac{c_Y(t-1)}{4}\|\va_{t-1}\|^2+a_t^2 \varsigma_t\geq \min\left(\frac{c_Y(t-1)}{4}\,,\varsigma_t\right)\|\va_t\|_2^2\,,
\end{align}
which proves the result. Otherwise, we obtain the inequality
\begin{align*}
\lim_{N\to\infty}\<Y_t\,\va_t,Y_t\,\va_t\>
&\geq a_t^2 \varsigma_t\ge \left(\frac{\varsigma_t\, c_Y(t-1)}{4(c'_t)^2+c_Y(t-1)}\right)\|\va_t\|_2^2\,,
\end{align*}
that completes the induction argument.
\end{itemize}

%
%
\section{A concentration estimate}

The following proposition follows from standard concentration-of-measure
arguments.
\begin{proposition}\label{propo:Concentration}
Let $V\subseteq\reals^m$ a uniformly random linear space of dimension
$d$. For $\lambda\in (0,1)$, let $P_{\lambda}$ denote the orthogonal
projector on the first $m\lambda$ coordinates of $\reals^m$.
Define $Z(\lambda) \equiv\sup\{ \|P_{\lambda}v\|\, :\; v\in V,\; \|v\|=1\}$.
Then, for any $\ve>0$ there exists $c(\ve)>0$ such that, for all $m$
large enough (and $d$ fixed)
\begin{eqnarray}
\prob\{|Z(\kappa)-\sqrt{\lambda}|\ge \ve\}\le e^{-m\,c(\ve)}\, .
\end{eqnarray}
\end{proposition}
\begin{proof}
Let $Q\in\reals^{m\times d}$ be a uniformly random orthogonal
matrix. Its image is a uniformly random subspace of $\reals^m$
whence the following equivalent characterization of $Z(\lambda)$
is obtained
\begin{eqnarray*}
Z(\lambda) \ed\sup\{ \|P_{\lambda}Qu\|\, :\; u\in S^d\}
\end{eqnarray*}
where $S^d\equiv\{x\in\reals^d\, :\, \|x\|=1\}$ is the $d$-dimensional
sphere, and $\ed$ denotes equality in distribution.

Let $N_d(\ve/2)$ be a $(\ve/2)$-net in $S_d$, i.e. a subset of
vectors $\{u^1,\dots,u^M\}\in S^d$ such that, for any $u\in S^d$,
there exists $i\in \{1,\dots,M\}$ such that $\|u-u^i\|\le \ve/2$.
It follows from a standard counting argument \cite{Ledoux} that there exists
an  $(\ve/2)$-net  of size $|N_d(\ve/2)|\equiv M\le (100/\ve)^d$.
Define
\begin{align*}
Z_{\ve/2}(\lambda) \equiv\sup\{ \|P_{\lambda}Qu\|\, :\; u\in N_d(\ve/2)\}\, .
\end{align*}
Since $u\mapsto P_{\lambda}Qu$ is Lipschitz with modulus $1$,
we have
\begin{align*}
\prob\{|Z(\kappa)-\sqrt{\lambda}|\ge \ve\}&\le
\prob\{|Z_{\ve/2}(\kappa)-\sqrt{\lambda}|\ge \ve/2 \}\\
&\le \sum_{i=1}^M\prob\{|\|P_{\lambda}Qu^i\|-\sqrt{\lambda}|\ge \ve/2 \}\, .
\end{align*}
But for each $i$,  $Qu^i$ is a uniformly random vector with norm $1$
in $\reals^m$. By concentration of measure in $S^m$
\cite{Ledoux}, there exists a function $c(\ve)>0$ such that,
for $x\in S^m$ uniformly random
\begin{align*}
\prob\big\{\big|\|P_{\lambda}x\|-\sqrt{\lambda}\big|\ge \ve/2 \big\}
\le e^{-m\, c(\ve)}\, .
\end{align*}
Therefore we get
\begin{align*}
\prob\{|Z(\kappa)-\sqrt{\lambda}|\ge \ve\}\le |N_{d}(\ve/2)|
 e^{-m\, c(\ve)}\le \Big(\frac{100}{\ve}\Big)^d\,  e^{-m\, c(\ve)}
\end{align*}
which is smaller than $e^{-mc(\ve)/2}$ for all $m$ large enough.
\end{proof}
%
%

\section{Useful reference material}

In this appendix we collect a few known results that are used
several times  in our proof. We also provide some pointers to
the literature.

\subsection{Equivalence of $\ell^2$ and $\ell^1$ norm on random vector spaces}

In our proof we make use of the following well-known result of Kashin in the theory of diameters of smooth functions \cite{Kashin1977}.
\begin{theorem}[Kashin 1977]\label{thm:Kashin}
For any positive number $\upsilon$ there exist a universal constant $c_\upsilon$ such that for any
$n\ge 1$, with probability at least $1-2^{-n}$, for a uniformly random subspace $V_{n,\upsilon}$ of dimension $\lfloor n(1-\upsilon)\rfloor$,
\[
\forall~x\in V_{n,\upsilon}:~~~ c_\upsilon\|x\|_2\le \f{1}{\sqrt{n}}\|x\|_1\,.
\]
\end{theorem}
%
%
\subsection{Singular values of random matrices}

We will repeatedly make use of limit behavior of extreme singular values of
random matrices. A very general result was proved in \cite{BaiYin}
(see also \cite{BaiSilverstein}).
\begin{theorem}[\cite{BaiYin}]
\label{prop:marchenko-pastur}
Let $A\in\reals^{n\times N}$ be a matrix with i.i.d. entries
such that $\E\{A_{ij}\}=0$, $\E\{ A_{ij}^2\}=1/n$, and $n=N\delta$.
Let $\smaxA$ be the largest singular value of $A$, and
$\hsigmamin(A)$ be its smallest non-zero singular value.
Then
\begin{eqnarray}
\lim_{N\to\infty}\smaxA &\asequal &\f{1}{\sqrt{\delta}}+1\, ,\\
\lim_{N\to\infty}\hsigmamin(A)
&\asequal &\f{1}{\sqrt{\delta}}-1\, .
\end{eqnarray}
\end{theorem}

We will also use the following fact that follows from
the standard singular value decomposition
\begin{eqnarray}
\min\big\{ \|Ax\|_2\, :\; x\in \ker(A)^\perp, \; \|x\|=1\big\}
= \sigma_{\rm min}(A)\, .
\end{eqnarray}
%
%
%
\subsection{Two Lemmas from \cite{BM-MPCS-2010}}
\label{sec:FromMPCS}

Our proof uses the results of \cite{BM-MPCS-2010}. We state
copy here the crucial technical lemma in that paper.
Notations refer to the general algorithm in Eq.~\eqref{eq:mpMain}.
General
state evolution defines quantities $\{\tau_t^2\}_{t\ge0}$ and $\{\sigma_t^2\}_{t\ge0}$ via
\begin{eqnarray}
\tau_{t}^2 = \E\big\{ g_t(\sigma_{t} Z,W)^2\big\}\, ,
\;\;\;\;\; \sigma_{t}^2 = \frac{1}{\delta}\,
\E\big\{ f_t(\tau_{t-1} Z,X_0)^2\big\}\, ,\label{eq:tau-sigma-recursion}
\end{eqnarray}
where $W\sim \pW$ and $X_0\sim p_{X_0}$ are independent of $Z\sim\normal(0,1)$
\begin{lemma}\label{lem:elephant}
Let $\{q_0(N)\}_{N\ge 0}$ and $\{A(N)\}_{N\ge 0}$ be, respectively,
a sequence of deterministic  initial conditions and a sequence of matrices $A\in\reals^{n\times N}$
indexed by $N$ with i.i.d. entries $A_{ij}\sim \normal(0,1/n)$.
Assume $n/N\to\delta\in (0,\infty)$. Consider deterministic sequences of vectors $\{x_0(N),w(N)\}_{N\ge 0}$, whose
empirical distributions converge weakly to  probability measures
$p_{X_0}$ and $\pW$ on $\reals$  with bounded $(2k-2)^{th}$ moment, and
assume:
\begin{itemize}
\item[(i)] $\lim_{N\to\infty}\E_{\empr_{x_0(N)}}(X_0^{2k-2})
=\E_{p_{X_0}}(X_0^{2k-2})<\infty$.
\item[(ii)] $\lim_{N\to\infty}\E_{\empr_{w(N)}}(W^{2k-2})=
\E_{\pW}(W^{2k-2})<\infty$.
\item[(iii)] $\lim_{N\to\infty} \E_{\empr_{q_0(N)}}( X^{2k-2})<\infty$.
\end{itemize}

Let $\{\sigma_t,\tau_t\}_{t\ge 0}$ be defined uniquely by the
recursion \eqref{eq:tau-sigma-recursion} with initialization
$\sigma_0^2=\delta^{-1}\lim_{n\to\infty}\<q^0,q^0\>$.  Then the following hold
for all $t\in\naturals\cup\{0\}$
\begin{itemize}
\item[$(a)$]
\begin{align}
h^{t+1}|_{\sigal{S}_{t+1,t}}&\deq \sum_{i=0}^{t-1}{\alpha_i}h^{i+1}+ {\tA}^*m_\perp^t+\tQ_{t+1}\order_{t+1}(1)\, ,\label{eq:main-lem-h-c}\\
b^t|_{\sigal{S}_{t,t}}&\deq \sum_{i=0}^{t-1}{\beta_i}b^{i} + {\tA}q_\perp^{t}+ \tM_t\order_t(1)\, ,\label{eq:main-lem-z-c}
\end{align}
where ${\tA}$ is an independent copy of $A$ and the matrix $\tQ_t$ ($\tM_t$) is such that its columns form an orthogonal basis for the column space of $Q_t$ ($M_t$) and $\tQ_t^*\tQ_t=N\,\identity_{t\times t}$ ($\tM_t^*\tM_t=n\,\identity_{t\times t}$).

\item[$(b)$] For all pseudo-Lipschitz functions $\phi_h,\phi_b:\reals^{t+2}\to\reals$ of order $k$
\begin{align}
\lim_{N\to\infty}\f{1}{N}\sum_{i=1}^N\phi_h(h_i^1,\ldots,h_i^{t+1},x_{0,i})
&\asequal\E\big\{\phi_h(\tau_0 Z_0,\ldots,\tau_tZ_{t},X_0)\big\}\, ,\label{eq:main-lem-h-a}\\
\lim_{n\to\infty}\f{1}{n}\sum_{i=1}^n\phi_b(b_i^0,\ldots,b_i^{t},w_i)
&\asequal\E\big\{\phi_b(\sigma_0\hat{Z}_0,\ldots,\sigma_t\hat{Z}_t,W)\big\}\,
,\label{eq:main-lem-z-a}
\end{align}
where $(Z_0,\ldots,Z_{t})$ and $(\hat{Z}_0,\ldots,\hat{Z}_{t})$
are two zero-mean gaussian vectors independent of $X_0$, $W$,
with $Z_i,\hat{Z}_i\sim \normal(0,1)$.

\item[$(c)$] For all $0\leq r,s\leq t$ the following equations hold and all limits exist, are bounded and have degenerate distribution
(i.e. they are constant random variables):
\begin{align}
\lim_{N\to\infty}\< h^{r+1},h^{s+1}\>&\asequal \lim_{n\to\infty}\< m^{r},m^{s}\>\, ,\label{eq:main-lem-h-b}\\
\lim_{n\to\infty}\< b^r,b^s\>&\asequal\f{1}{\delta}\lim_{N\to\infty}\< q^r,q^s\>\, .\label{eq:main-lem-z-b}
\end{align}

\item[$(d)$] For all $0\leq r,s\leq t$, and for any Lipschitz function
$\varphi:\reals^2\to\reals$ , the following equations hold and all limits exist, are bounded and have degenerate distribution (i.e. they are constant random variables):
\begin{align}
\lim_{N\to\infty}\< h^{r+1}, \varphi(h^{s+1},x_0)\>&\asequal \lim_{N\to\infty}\< h^{r+1},h^{s+1}\> \< \varphi'(h^{s+1},x_0)\>,\label{eq:main-lem-h-d}\\
\lim_{n\to\infty} \< b^{r},\varphi(b^s,w)\>&\asequal \lim_{n\to\infty}\< b^r,b^s\> \< \varphi'(b^s,w)\>\,. \label{eq:main-lem-z-d}
\end{align}
Here $\varphi'$ denotes derivative with respect to the first coordinate of $\varphi$.

\item[$(e)$] For $\ell = k-1$, the following hold
almost surely
\begin{align}
\lim\sup_{N\to\infty}\f{1}{N}\sum_{i=1}^N (h_i^{t+1})^{2\ell}&<\infty\,,\label{eq:main-lem-h-e}\\
\lim\sup_{n\to\infty}\f{1}{n}\sum_{i=1}^n (b_i^t)^{2\ell}&<\infty.\label{eq:main-lem-z-e}
\end{align}

\item[$(f)$] For all $0\leq r\le t$:
\begin{align}
\lim_{N\to\infty}\f{1}{N}\< h^{r+1},q^0\>&\asequal0\,.\label{eq:<ht,q0>=0}
\end{align}

\item[$(g)$] For all $0\leq r\leq t$ and $0\leq s \leq t-1$ the following limits exist, and there exist strictly positive constants $\lbq_r$ and $\lbm_s$ (independent of $N$, $n$) such that almost surely

\begin{align}
\lim_{N\to\infty}\< q^{r}_{\perp},q^{r}_{\perp}\>&>\lbq_r\,,\label{eq:main-lem-q-g}\\
\lim_{n\to\infty}\< m^{s}_\perp,m^{s}_\perp\>&>\lbm_s\,.\label{eq:main-lem-m-g}
\end{align}
\end{itemize}
\end{lemma}

It is also useful to recall some simple  properties of gaussian random
matrices.
\begin{lemma}\label{lem:prop-Gaussian-matrix} For any deterministic $u\in\reals^N$ and $v\in\reals^n$ with $\|u\|=\|v\|=1$ and a gaussian matrix ${\tA}$ distributed as $A$ we have
\begin{itemize}
\item[(a)] $v^* {\tA} u\deq Z/\sqrt{n}$ where $Z\sim \normal(0,1)$.

\item[(b)] $\lim_{n\to\infty} \|{\tA}u\|^2 = 1$ almost surely.

\item[(c)] Consider, for $d\le n$, a $d$-dimensional subspace $W$ of
$\reals^n$, an orthogonal basis $w_1,\ldots,w_d$ of $W$ with $\|w_i\|^2=n$ for $i=1,\ldots,d$, and the orthogonal projection $P_W$ onto $W$. Then for $D=[w_1|\ldots|w_d]$, we have $P_W A u\deq Dx$ with $x\in\reals^d$ that satisfies: $\lim_{n\to\infty}\|x\|\asequal 0$
(the limit being taken with $d$ fixed). Note that $x$ is $\order_d(1)$ as well.
\end{itemize}
\end{lemma}

%
%
\bibliographystyle{amsalpha}

\providecommand{\bysame}{\leavevmode\hbox to3em{\hrulefill}\thinspace}
\providecommand{\MR}{\relax\ifhmode\unskip\space\fi MR }
\providecommand{\MRhref}[2]{%
  \href{http://www.ams.org/mathscinet-getitem?mr=#1}{#2}
}
\providecommand{\href}[2]{#2}

\newpage

{\bf Mohsen Bayati} is an assistant professor of operations and information
technology at Stanford university Graduate School of Business.  Mohsen
received his PhD in Electrical Engineering from Stanford University in 2007.
His dissertation was on machine learning and modeling aspects of large-scale
networks. During the summers of 2005 and 2006 he interned at IBM Research
and Microsoft Research respectively.  He was a Postdoctoral Researcher with
Microsoft Research from 2007 to 2009 working mainly on applications of
machine learning and optimization methods in healthcare and online
advertising. In particular, he focused on hospital readmissions.  He has
been a Postdoctoral Scholar at Stanford University from 2009 to 2011 with a
research focus in high-dimensional statistical data-mining.

\vspace{1cm}

{\bf Andrea Montanari}  is an associate professor
in the Departments of Electrical Engineering and of Statistics, Stanford
University.
He received the Laurea degree in physics in 1997, and the
Ph.D. degree in theoretical physics in 2001, both from Scuola Normale Superiore,
Pisa, Italy.
He has been a Postdoctoral Fellow with the Laboratoire de Physique
Th\'eorique of Ecole Normale Sup\'erieure (LPTENS), Paris, France, and the
Mathematical Sciences Research Institute, Berkeley, CA. Since 2002, he has
been Charg\'e de Recherche (a research position with Centre National de la
Recherche Scientifique, CNRS) at LPTENS. In September 2006, he joined the
faculty of Stanford University.
Dr. Montanari was coawarded the ACM SIGMETRICS Best Paper Award in
2008. He received the CNRS Bronze Medal for Theoretical Physics in 2006 and
the National Science Foundation CAREER award in 2008.
His research focuses on algorithms on graphs, graphical models, statistical inference
and estimation.

\end{document}